\documentclass[twoside, 12pt]{amsart}
\usepackage{amsmath, amssymb, amsthm, amscd,enumerate}
\usepackage{epsfig, eucal,float}

\numberwithin{equation}{section}

\newtheorem{theorem}{Theorem}[section]

\newtheorem{lemma}[theorem]{Lemma}

\newtheorem{proposition}[theorem]{Proposition}

\newtheorem{definition}[theorem]{Definition}
\newtheorem{remark}[theorem]{Remark}

\newtheorem{convention}[theorem]{Convention}

\newcommand{\NN}{\mathbb{N}}
\newcommand{\ZZ}{\mathbb{Z}}
\newcommand{\QQ}{\mathbb{Q}}
\newcommand{\QQQ}{\hat{\mathbb{Q}}}
\newcommand{\RR}{\mathbb{R}}
\newcommand{\RRR}{\hat{\mathbb{R}}}
\newcommand{\CC}{\mathbb{C}}
\newcommand{\CCC}{\hat{\mathbb{C}}}

\newcommand{\HH}{\mathbb{H}}
\newcommand{\PSL}{\operatorname{PSL}}
\newcommand{\LL}{\mathcal{L}}
\newcommand{\trg}{\mbox{$\mathrm{trg}$}}
\newcommand{\CT}{\mbox{$\kappa$}}
\newcommand{\CTM}{\mbox{$\kappa_{\mathcal{P}}$}}
\newcommand{\OO}{\mathcal{O}}
\newcommand{\End}{\mathcal{E}}
\newcommand{\Parabolic}{\mathbb{P}}
\newcommand{\Partition}{\mathcal{P}}
\newcommand{\Sphere}{\mathbb{S}_{\mathcal{P}}}
\newcommand{\CCM}{\mathbb{C}_{\mathcal{P}}}
\newcommand{\inftym}{\infty_{\mathcal{P}}}
\newcommand{\spider}{\mathfrak{s}}
\newcommand{\leg}{\ell}
\newcommand{\FF}{\mathcal{F}}
\newcommand{\Tbundle}{\mbox{$T_{\varphi}$}}
\newcommand{\Obundle}{\mbox{$\mathcal{O}_{\varphi}$}}
\newcommand{\cellcomplex}{\mbox{$CW\Delta$}}

\newcommand{\inn}[1]{\ensuremath{\begin{aligned}[b] &{}\\[-.7cm] &{}\mkern3mu\scriptscriptstyle{\circ}\\[.15cm]
                                  \end{aligned}}\mkern-12mu#1}

\renewcommand{\le}{\leqslant}
\renewcommand{\ge}{\geqslant}

\begin{document}
\title{\mbox{{\hskip-.9cm} On hyperbolic once-punctured-torus bundles III:}\\
\mbox{ {\hskip-1.1cm} \small Comparing two tessellations of the complex plane}}

\author{Warren Dicks and Makoto Sakuma}

\thanks{The first author is supported jointly
by the MEC (Spain) and the EFRD~(EU) through Projects MTM2006-13544 and MTM2008-01550.}
\thanks{
The second author is supported by JSPS Grants-in-Aid 18340018, and
is partially supported by JSPS Core-to-Core Program 18005.}

\address{Warren Dicks, Departament de Matem\`atiques, Universitat Aut\`o\-no\-ma
de Barcelona, 08193 Bellaterra (Barcelona), Spain.}
\email{dicks@mat.uab.cat}

\address{Makoto Sakuma,
Department of Mathematics,
Graduate School of Science,
Hiroshima University,
Kagamiyama 1-3-1,
Higashi-Hiroshima,
Hiroshima 739-8526,
Japan}
\email{sakuma@math.sci.hiroshima-u.ac.jp}

\begin{abstract}
To each once-punctured-torus bundle, $\Tbundle$, over the circle
with pseudo-Anosov monodromy $\varphi$, there are associated two
tessellations of the complex plane: one,  $\Delta(\varphi)$,  is
(the projection from $\infty$ of) the triangulation of  a horosphere
at $\infty$  induced by the canonical decomposition into ideal
tetrahedra, and the other,  $CW(\varphi)$, is a fractal tessellation
given by the Cannon-Thurs\-ton map of the fiber group switching back
and forth between gray and white each time it passes through
$\infty$. In this paper, we study the relation between
$\Delta(\varphi)$ and $CW(\varphi)$.
\end{abstract}

\maketitle

\begin{flushright}
\textit{Dedicated\, to\, Prof.\ \,Akio \,Kawauchi\linebreak
on the occasion of his 60th birthday}
\end{flushright}

\section{Introduction}
\label{sec.introduction}
Let $\varphi$ be a pseudo-Anosov homeomorphism
of the once-punctured torus $T:=(\RR^2-\ZZ^2)/\ZZ^2$
and let $\Tbundle:=T\times\RR/(x,t)\sim (\varphi(x),t+1)$
be the bundle over the circle with fiber $T$
and monodromy $\varphi$.
By Thurston's uniformization theorem for surface bundles
(\cite{Thurston, Otal}),
$\Tbundle$ admits a complete hyperbolic structure of finite volume.
Since $\Tbundle$ has a single torus cusp,
$\Tbundle$ admits the
\textit{canonical decomposition}
into ideal tetrahedra
which is dual to the Ford domain (\cite{Epstein-Penner, Weeks}).
The complete hyperbolic structure and the canonical
decomposition of $\Tbundle$
were constructed by
J\o rgensen
in his famous unfinished work
\cite{Jorgensen}, and rigorous
treatments
of (part of) his results were given
in
\cite{Akiyoshi, ASWY, Gueritaud1, Gueritaud2,Lackenby,
Parker}.

The canonical decomposition induces a triangulation of
any peripheral torus,
which in turn lifts to a triangulation, $\Delta(\varphi)$,
of the universal covering of the
peripheral torus.
We may assume that the ideal point $\infty$ of the upper-half-space
model
$\HH^3=\CC\times\RR_{+}$ of
hyperbolic space is a parabolic fixed point of the Kleinian group
$\Gamma\cong\pi_1(\Tbundle)$ uniformizing $\Tbundle$.
Then we may regard $\Delta(\varphi)$   as a triangulation of
a horosphere at $\infty$, and then,
by projection from $\infty$,
as a triangulation of the complex
plane $\CC$ which is invariant by the stabilizer
$\Gamma_{\infty}\cong\ZZ^2$ of $\infty$ in $\Gamma$.

On the other hand, by the first author's joint work \cite{Cannon-Dicks2}
with J. W. Cannon,
the complex plane admits
a rather different
$\Gamma_{\infty}$-invariant
(fractal) tessellation, $CW(\varphi)$,
which naturally arises from the Cannon-Thurston map
associated to $\Tbundle$.
This is intimately related with
the fractal domain in
the first author's previous joint work
with R.C. Alperin and J. Porti \cite{Alperin-Dicks-Porti},
which plays a key role in their construction of
the Cannon-Thurston map
associated to the Gieseking manifold,
a quotient of
the simplest hyperbolic punctured-torus bundle.

Since both $\Delta(\varphi)$ and $CW(\varphi)$
are $\Gamma_{\infty}$-invariant tessellations
of the complex plane which naturally
arise from the punctured-torus bundle $\Tbundle$,
it is reasonable to expect some nice relation between them.
The purpose of this paper
is to show that this is actually the case
(see Figure \ref{fig.RLLRRRLLLL} and Theorem \ref{maintheorem}).
In fact, we show that
$\Delta(\varphi)$ and $CW(\varphi)$
share the same vertex set and that
the combinatorial structure of $CW(\varphi)$
can be
recovered from that of $\Delta(\varphi)$
and vice-versa.
To be more precise,
$\Delta(\varphi)$ is endowed with a structure of
a \lq\lq layered simplicial complex'',
which reflects
the bundle structure of $\Tbundle$
(Section \ref{sec.cusptriangulation}),
whereas $CW(\varphi)$ is endowed with a structure
of a
\lq\lq colored CW-complex''
(Section \ref{sec.fractal-tessellation}),
which reflects
the way that the Cannon-Thurston map
fills in the Riemann sphere.
We show in Theorem \ref{maintheorem} that
$\Delta(\varphi)$ with the layered structure
(combinatorially) determines
$CW(\varphi)$
with the colored structure,
and vice versa.
The theorem is proved by constructing a certain CW-decomposition of the
complex plane which serves as a common parent
of the two tessellations
(see Definition \ref{definition.parents-complex} and
Proposition \ref{prop.quotient-cellcomplex}).

\begin{figure}[p]
\begin{center}
\setlength{\unitlength}{1truecm}
\begin{picture}(12,17.3)
\put(0,0){\epsfig{file= 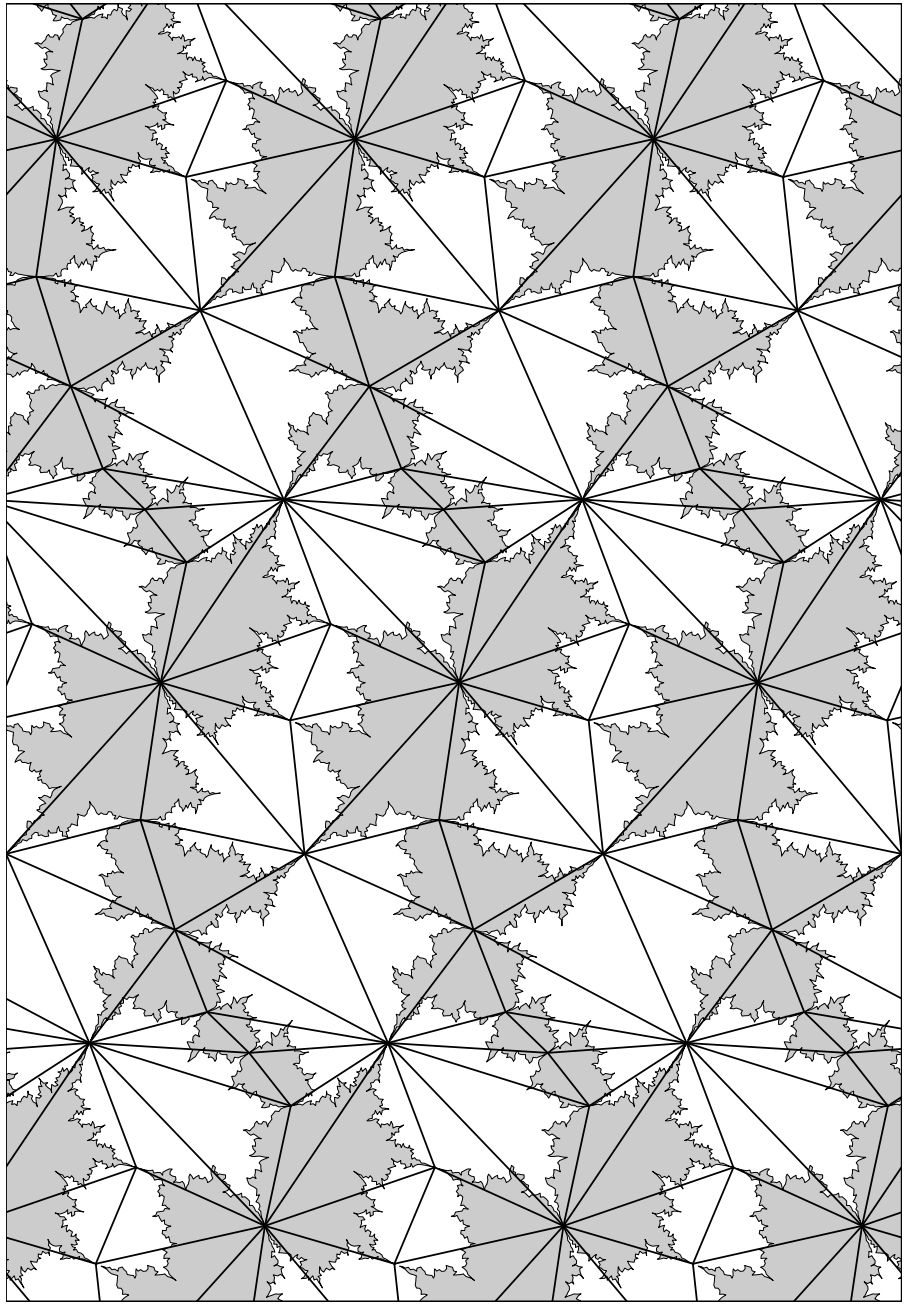,width=12truecm}}
\end{picture}
\caption{Projected-horosphere
triangulation \, $\Delta(\varphi)$
and fractal tessellation $CW(\varphi)$,
for $\varphi\!=\!RLLRRRLLLL.$\newline
The straight line segments etch  $\Delta(\varphi)$
while the fractal lines etch  $CW(\varphi)$.}\label{fig.RLLRRRLLLL}
\end{center}
\end{figure}

This paper is organized as follows.
In Section \ref{sec.orbifold},
we recall basic facts concerning the orbifold fundamental group
$\pi_1(\OO)$ of the $(2,2,2,\infty)$-orbifold,
$\OO$, obtained as the quotient of $T$
by the hyper-elliptic involution.
The contents in this section
give
the common language to describe the combinatorial structures
of $\Delta(\varphi)$ and $CW(\varphi)$.
In Section \ref{sec.farey},
we describe the normal form of the pseudo-Anosov map $\varphi$
and fix a convention (Convention \ref{convention.Farey}),
which we employ throughout the paper.
In Section \ref{sec.representation},
we describe the \lq\lq type-preserving''
$\PSL(2,\CC)$-representations of $\pi_1(\OO)$,
and fix notation
for the punctured-torus bundle $\Tbundle$
and its natural quotient $\Obundle$.
In Section \ref{sec.cusptriangulation},
we recall the combinatorial description
of the canonical decomposition of $\Tbundle$,
introduce the \lq\lq layered structure'' of $\Delta(\varphi)$
(Definition \ref{layered complex}),
and give a combinatorial description
of $\Delta(\varphi)$ in terms of the language
prepared in Section \ref{sec.orbifold}
(Theorem \ref{thm.canonical-decomposition} and
Proposition \ref{prop.model-Delta}).
In Section \ref{sec.Cannon-Thurston},
we recall the combinatorial description of the Cannon-Thurston map
associated to $\Tbundle$,
which was established by Bowditch \cite{Bowditch}
(Theorem \ref{thm.Bowditch}).
In Section \ref{sec.fractal-tessellation},
we recall the fractal tessellation $CW(\varphi)$
introduced in \cite{Cannon-Dicks2}, and, in Theorem \ref{Th.CDTessellation},
we give a combinatorial description of $CW(\varphi)$
in terms of the common language developed in Section \ref{sec.orbifold}.
Finally, in Section \ref{sec.statement},
we state the main theorem (Theorem \ref{maintheorem})
and give a proof of the theorem.

\section{The orbifold $\OO$ and its fundamental group}
\label{sec.orbifold}

The punctured torus $T=(\RR^2-\ZZ^2)/\ZZ^2$
admits the \textit{hyper-elliptic involution},
induced by the linear
automorphism
$x\mapsto -x$
of $\RR^2-\ZZ^2$.
The quotient of $T$ by the involution is the
$(2,2,2,\infty)$-orbifold $\OO$,
i.e., the orbifold with underlying space
a once-punctured sphere and with three cone points
of index $2$.
The orbifold fundamental group $\pi_1(\OO)$ is defined to be the
covering transformation group of the universal cover
$\tilde\OO$ of $\OO$.
Since $T$ is a 2-fold
(branched)
covering of $\OO$,
$\tilde\OO$ is identified with the universal cover
$\tilde T$ of $T$,
and $\pi_1(T)$ is a subgroup of $\pi_1(\OO)$ of index $2$.

The group $\pi_1(\OO)$ has the following presentation:
\begin{equation}
\pi_1(\OO) = \langle A,B,C
\; |\;
A^2=B^2=C^2=1\rangle.
\end{equation}
Set $D:=CBA$. Then $D$ (resp. $D^2$)
is a peripheral element of $\pi_1(\OO)$ (resp. $\pi_1(T)$),
namely it is represented by a simple loop around the puncture of
$\OO$ (resp. $T$).
We call $D$ the \textit{distinguished element}.

By picking a complete hyperbolic structure of $\OO$ (and hence
of $T$), we identify $\tilde\OO=\tilde T$ with (the upper-half-space
model of) the hyperbolic plane $\HH^2=\{z\in\CC\mid\Im(z)>0\}$, and
identify $\pi_1(\OO)$ with a Fuchsian group (see Figure
\ref{fig.fuchsian-group}).

\begin{figure}[t!]
\begin{center}
\setlength{\unitlength}{1truecm}
\begin{picture}(7.3, 3.6)
\put(0.2,0.2){\epsfig{file= 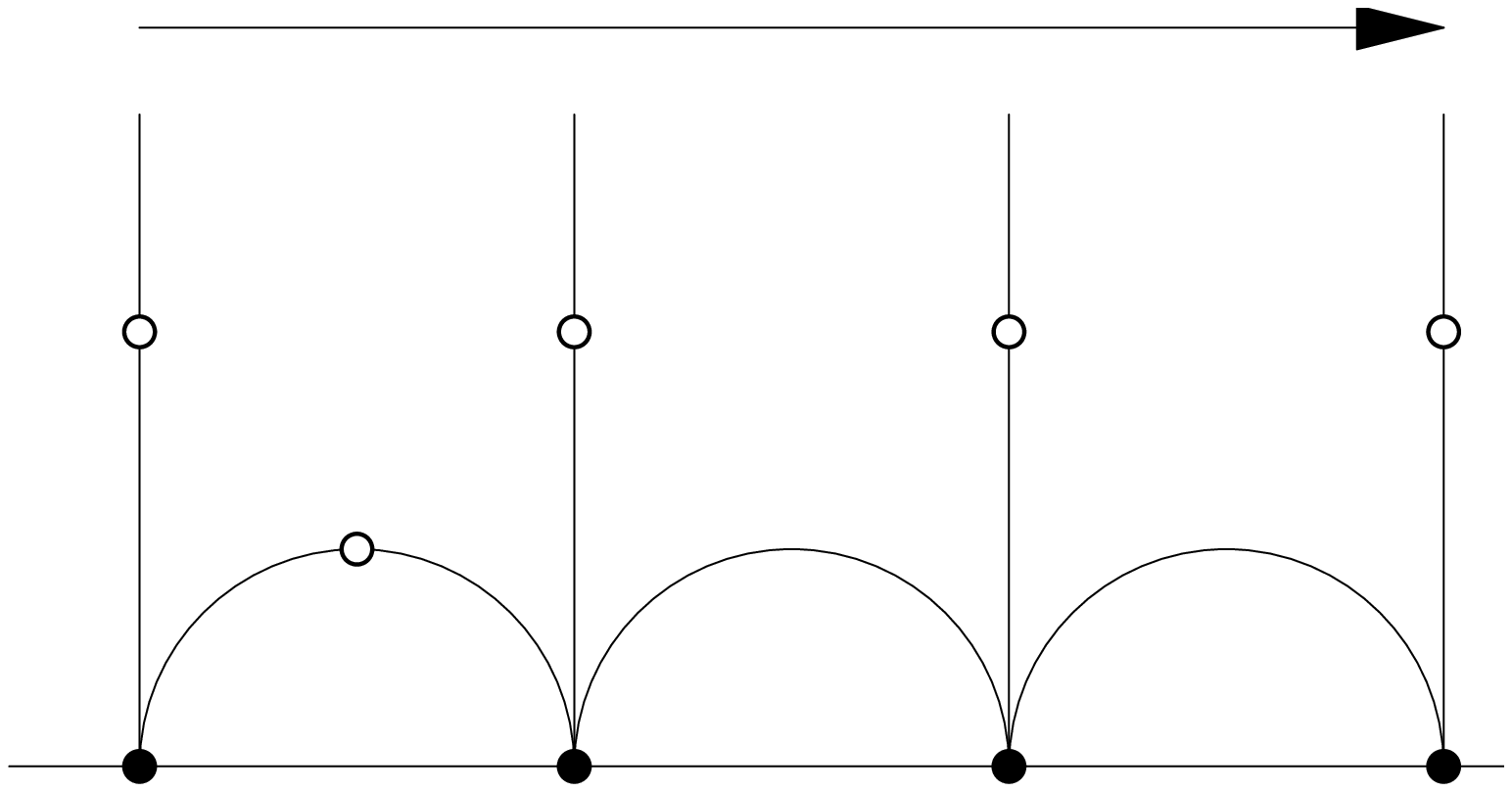, width = 7.3truecm}}
\put(.6,2.2){\makebox(0,0)[c]{$A$}}
\put(2.4,2.2){\makebox(0,0)[c]{$B$}}
\put(4.25,2.2){\makebox(0,0)[c]{$C$}}
\put(7.05,2.2){\makebox(0,0)[c]{$\inn{D}(A)$}}
\put(1.9,1){\makebox(0,0)[c]{$\inn{B}(C)$}}
\put(2.8,3.9){\makebox(0,0)[c]{$D$}}
\put(.9,0){\makebox(0,0)[c]{$A(\infty)$}}
\put(2.6,0){\makebox(0,0)[c]{$B(\infty)$}}
\put(4.6,0){\makebox(0,0)[c]{$C(\infty)$}}
\put(6.7,0){\makebox(0,0)[c]{$DA(\infty)$}}
\end{picture}
\caption{Fuchsian group $\langle A,B,C\rangle$}\label{fig.fuchsian-group}
\end{center}
\end{figure}

Then $D$ is identified with the following parabolic transformation
having the ideal point $\infty$ of $\HH^2$
as the parabolic fixed point.
\begin{equation}
\label{NormalizingD0}
D(z)=z+1.
\end{equation}
Then
the points $A(\infty)$, $B(\infty)$ and $C(\infty)$ lie on $\RR$
from left to right in this order.
After a coordinate change, we may assume that
the images of the three geodesics
joining $\infty$ with these three points, in the
universal abelian cover $\RR^2-\ZZ^2$ of $T$,
are open arcs of slopes $0$, $1$ and $\infty$,
joining the puncture $(0,0)$ with $(1,0)$, $(1,1)$ and $(0,1)$,
respectively.
Thus the images of these three geodesics in $T$
are mutually disjoint arcs properly embedded in $T$,
which divide $T$ into two ideal triangles, and thus
they determine an ideal triangulation of $T$.

We now recall the well-known correspondence between the ideal
triangulations of $T$ and the \textit{Farey triangles}. The
\textit{Farey tessellation} is the tessellation of the hyperbolic
plane $\HH^2$ obtained from the ideal triangle $\langle
0,1,\infty\rangle$ by successive reflection in its edges. The vertex
set of the Farey tessellation is equal to $\QQQ:=\QQ\cup\{1/0\}
\subset \partial \HH^2$ and each vertex $r$ determines a properly
embedded arc $\beta_r$ in $T$ of \textit{slope} $r$, i.e., the arc
in $T$ obtained as the image of the straight arc of slope $r$ in
$\RR^2-\ZZ^2$ joining punctures. If $\sigma=\langle
r_0,r_1,r_2\rangle$ is a Farey triangle, i.e., a triangle in the
Farey tessellation, then the arcs $\beta_{r_0}$, $\beta_{r_1}$ and
$\beta_{r_2}$ are mutually disjoint and they determine an ideal
triangulation, $\trg(\sigma)$, of $T$. In the following we assume
that the orientation of $\sigma=\langle r_0,r_1,r_2\rangle$ is
coherent with the orientation of the Farey triangle  $\langle
0,1,\infty\rangle$, where the orientation is determined by the order
of the vertices. Then the oriented simple loop in $T$ around the
puncture representing $D^2$ meets the edges of $\trg(\sigma)$ of
slopes $r_0,r_1,r_2$ in this cyclic order, for every Farey triangle
$\sigma=\langle r_0,r_1,r_2\rangle$.

By using the above notation,
the generators $A$, $B$ and $C$ are described as follows.
Consider the ideal triangulation $\trg(\sigma)$ of $T$
determined by the Farey triangle $\sigma=\langle 0,1,\infty\rangle$.
It lifts to a $\pi_1(\OO)$-invariant tessellation of
the universal cover $\tilde T=\HH^2$.
Let $\{e_j\}_{j\in\ZZ}$ be the edges of the tessellation emanating from
the ideal vertex $\infty$, lying in $\HH^2$ from left to right in this order.
For each $e_j$,
there is a unique order $2$ element, $P_j\in \pi_1(\OO)$
which inverts $e_j$.
We may assume after a shift of indices that
$e_{3j}$, $e_{3j+1}$ and $e_{3j+2}$
project to the arcs in $T$ of slopes $0$, $1$ and $\infty$,
respectively, for every $j\in\ZZ$.
Then any triple of consecutive elements
$(P_{3j}, P_{3j+1}, P_{3j+2})$ serves as $(A,B,C)$.
Throughout this paper, $(A,B,C)$ represents the triple of
specific elements of $\pi_1(\OO)$ obtained in this way.
We call $\{P_j\}_{j\in\ZZ}$
the \textit{sequence of elliptic generators}
associated with the Farey triangle $\sigma$.

The above construction works for every Farey triangle
$\sigma=\langle r_0,r_1,r_2\rangle$,
and the sequence of elliptic generators associated with it is defined.
(Here we use the assumption that the orientation of
$\langle r_0,r_1,r_2\rangle$ is coherent with the orientation of
$\langle 0,1,\infty\rangle$.)
Any triple of three consecutive elements
in a sequence of elliptic generators is called an
\textit{elliptic generator triple}.
A member, $P$, of an elliptic generator triple is called
an \textit{elliptic generator},
and its \textit{slope} $s(P)\in\QQQ$ is defined
to be the slope of the arc in $T$ obtained as the image
of the geodesic $\langle \infty, P(\infty)\rangle$.
(Here it should be noted that $\infty$ is the parabolic fixed point
of the distinguished element $D$.)
For example, we have
\begin{equation}
(s(A),s(B),s(C))=(0,1,\infty).
\end{equation}
When we say that $\{P_j\}_{j\in\ZZ}$
is the sequence of elliptic generators
associated with a Farey triangle
$\sigma=\langle r_0,r_1,r_2\rangle$,
we always assume that
\[
(s(P_{3m}),s(P_{3m+1}),s(P_{3m+2}))=(r_0,r_1,r_2).
\]
Thus the index $j$ is well defined modulo a shift by a multiple of
$3$. We summarize the properties of elliptic generators (cf.
\cite[Section 2.1]{ASWY}).
We shall use the following non-standard notation.
\begin{align}
\label{eq:conjugation}
&\text{For elements $X$, $Y$ of a group $G$,  $\inn{X}(Y)$ denotes $XYX^{\scriptscriptstyle -1}$.}
\end{align}
We view $\inn{X}$ as an element of the automorphism group of $G$.

\begin{proposition}
{\normalfont(1)}  Let $\{P_j\}_{j\in\ZZ}$
be
the sequence of elliptic generators
associated with a Farey triangle $\sigma$.
Then the following hold for every $j\in\ZZ$.
\begin{enumerate}[\normalfont (i)]
\item
$
\pi_1(\OO) \cong \langle P_j, P_{j+1}, P_{j+2}
\; |\;
P_j^2=P_{j+1}^2=P_{j+2}^2=1\rangle.
$
\item
$P_{j+2}P_{j+1}P_j$ is equal to the distinguished element $D$
of $\pi_1(\OO)$.
\item
With the notation of~\eqref{eq:conjugation}, $P_{j+3m}= \inn{D^m}(P_j)$
for every $m \in \ZZ$.
\item
$\langle s(P_j), s(P_{j+1}), s(P_{j+2})\rangle$
is a Farey triangle and its orientation is coherent with
$\langle 0,1,\infty\rangle$.
\end{enumerate}

{\normalfont(2)} Let $P$ and $P'$ be elliptic generators of the same slope.
Then $P'= \inn{D^m}(P)$
for some $m\in\ZZ$.
Let $\sigma=\langle r_0,r_1,r_2 \rangle$ and
$\sigma'=\langle r_0',r_1',r_2' \rangle$ be Farey triangles
sharing the edge
$\langle r_0,r_1\rangle = \langle r_0',r_2' \rangle$,
and let $\{P_j\}$ and $\{P_j'\}$, respectively,
be the sequences of elliptic generators associated with
$\sigma$ and $\sigma'$.
Then the following identity holds after
a shift of indices by a multiple of $3$
(see Figure \ref{figure.EGS}).
\[
(P'_{3j},P'_{3j+1},P'_{3j+1})
=
(P_{3j},\inn{P_{3j+1}}(P_{3j+2}),P_{3j+1}).
\]
\end{proposition}

\begin{figure}[t!]
\begin{center}
\setlength{\unitlength}{1truemm}
\begin{picture}(100,35)(0,-5)
\put(0,20){\makebox(0,0)[c]{$P_{0}'$}}
\put(0,15){\makebox(0,0)[c]{$\scriptstyle\|$}}
\put(0,10){\makebox(0,0)[c]{$P_{0}$}}
\put(10,0){\makebox(0,0)[c]{$P_{1}$}}
\put(17,20){\makebox(0,0)[c]{$P_{1}'$}}
\put(17,15){\makebox(0,0)[c]{$\scriptstyle\|$}}
\put(17,10){\makebox(0,0)[c]{$P_{2}$}}
\put(27,30){\makebox(0,0)[c]{$P_{2}'$}}
\put(38,20){\makebox(0,0)[c]{$P_{3}'$}}
\put(38,15){\makebox(0,0)[c]{$\scriptstyle\|$}}
\put(38,10){\makebox(0,0)[c]{$P_{3}$}}
\put(48,0){\makebox(0,0)[c]{$P_{4}$}}
\put(56,20){\makebox(0,0)[c]{$P_{4}'$}}
\put(56,15){\makebox(0,0)[c]{$\scriptstyle\|$}}
\put(56,10){\makebox(0,0)[c]{$P_{5}$}}
\put(75,20){\makebox(0,0)[c]{$s_{0}'$}}
\put(75,15){\makebox(0,0)[c]{$\scriptstyle\|$}}
\put(75,10){\makebox(0,0)[c]{$s_{0}$}}
\put(97,20){\makebox(0,0)[c]{$s_{1}'$}}
\put(97,15){\makebox(0,0)[c]{$\scriptstyle\|$}}
\put(97,10){\makebox(0,0)[c]{$s_{2}$}}
\put(10,7){\makebox(0,0)[c]{$\circlearrowleft$}}
\put(28,25){\makebox(0,0)[c]{$\circlearrowright$}}
\put(48,7){\makebox(0,0)[c]{$\circlearrowleft$}}
\put(85.5,20){\makebox(0,0)[c]{$\circlearrowleft$}}
\put(85.5,10){\makebox(0,0)[c]{$\circlearrowleft$}}
\put(86,32){\makebox(0,0)[c]{$s'_{2}$}}
\put(86,-2){\makebox(0,0)[c]{$s_{1}$}}
\put(77,0){\epsfig{file= 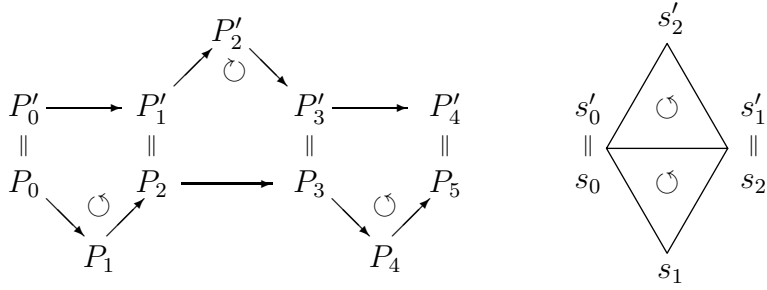, width = 1.7truecm}}
\thinlines
\put(3,20){\vector(1,0){10}}
\put(3,8){\vector(1,-1){5}}
\put(11,3){\vector(1,1){5}}
\put(20,23){\vector(1,1){5}}
\put(30,28){\vector(1,-1){5}}
\put(21,10){\vector(1,0){12}}
\put(41,20){\vector(1,0){10}}
\put(41,8){\vector(1,-1){5}}
\put(49,3){\vector(1,1){5}}
\end{picture}
\caption{Adjacent sequences of elliptic generators.
The symbol $\circlearrowleft$, resp. $\circlearrowright$, indicates
a triangle in which coherent reading of the vertices is  counter-clockwise, resp. clockwise. } \label{figure.EGS}
\end{center}
\end{figure}

The last assertion of the above proposition motivates us to
define
the \textit{right automorphism} $R$ and
the \textit{left automorphism} $L$ of $\pi_1(\OO)$
by the following rule:
\begin{equation}
\label{DefiningRL}
R:(A,B,C)\mapsto (A, \inn{B}(C),B),
\quad
L:(A,B,C)\mapsto (B, \inn{B}(A),C),
\end{equation}

The proof of the following lemma is straightforward.

\newpage

\begin{lemma}
\label{lemma.RL}
Let $\sigma_0=\langle 0,\infty,-1\rangle$ and
$\sigma_1=\langle 0,1,\infty\rangle$.
Then the following hold.
\begin{enumerate}[\normalfont (1)]
\item   $(A,B,C)$ is an elliptic generator triple
associated with $\sigma_1$.
Moreover the sequence of elliptic generators
associated with $\sigma_1$ is as follows:
\[
\cdots, \,
\inn{D^{\scriptscriptstyle -1}}(A), \, \inn{D^{\scriptscriptstyle-1}}(B), \, \inn{D^{\scriptscriptstyle-1}} (C), \,
A, \, B, \, C, \,
\inn{D}(A), \,  \inn{D}(B), \,  \inn{D}(C), \, \cdots.
\]
\item
$(A,C,\inn{C}(B))$
is an elliptic generator triple associated with $\sigma_0$.
\linebreak
Moreover the sequence of elliptic generators
associated with $\sigma_0$ is as follows:
\[
\cdots, \,
\inn{D^{\scriptscriptstyle -1}}(A), \, \inn{D^{\scriptscriptstyle -1}}(B),  \,
\inn{D^{\scriptscriptstyle -1}}\inn{C}(B) = \inn{A}(B), \,
A,  \, C, \,  \inn{C}(B), \,
\inn{D}(A),    \, \cdots.
\]
\item Both $R$ and $L$
map
the sequence of elliptic generators
associated with $\sigma_0$ to that associated with $\sigma_1$.
In fact, we have:
\[
R(A,C, \inn{C}(B))=(A,B,C), \quad
L(A,C, \inn{C}(B))=(B,C, \inn{D}(A)).
\]
Moreover, they differ only by post composition of a shift of indices
of elliptic generators associated with $\sigma_1$. To be precise,
$LR^{-1}(P_j)=P_{j+1}$, where $\{P_j\}_{j\in\ZZ}$ is the sequence of
elliptic generators associated with $\sigma_1$.
In particular, $(LR^{-1})^3 = \inn{D}$.
\end{enumerate}
\end{lemma}

Since $R$ and $L$ preserve the distinguished element $D$,
they map elliptic generators to elliptic generators.
Moreover, if $P$ and $P'$ are elliptic generators of the same slope,
then $R(P)$ and $R(P')$
(resp. $L(P)$ and $L(P')$) have the same slope.
Thus $R$ and $L$
act
on the set $\QQQ$
of
slopes
of the elliptic generators.
The action $R_*$ (resp.  $L_*$)
of $R$ (resp. $L$) on $\QQQ$
induces an automorphism of the Farey tessellation which acts as a
one-unit shift on the bi-infinite sequence of triangles incident on
the vertex $0$ (resp. $\infty$), and this shift can be thought of as
rotation to the right (resp. left). (See Figure \ref{fig.RLaction}.)

\begin{figure}[t!]
\begin{center}
\setlength{\unitlength}{1truecm}
\begin{picture}(6,6)
\put(0,0){\epsfig{file= 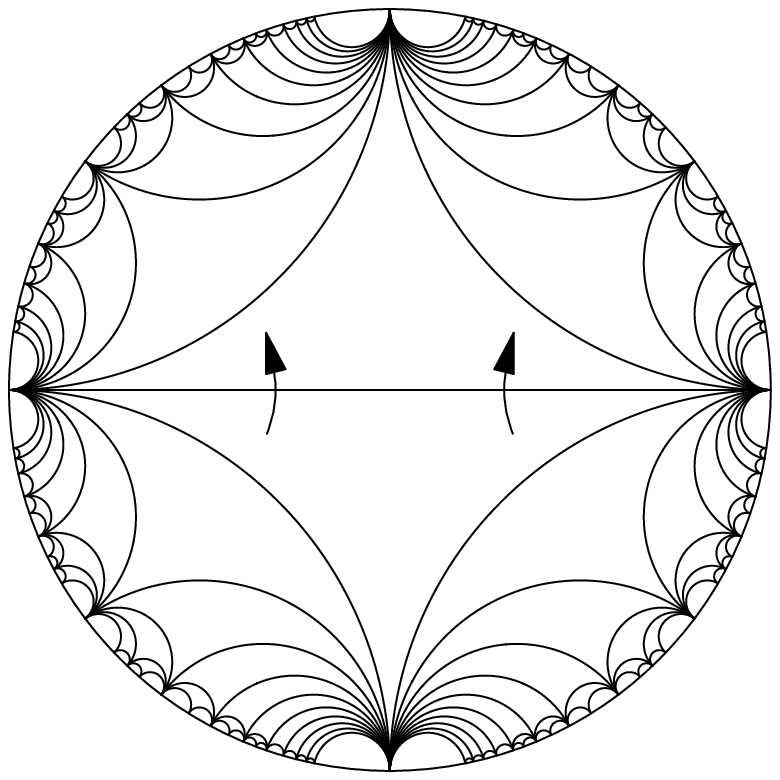, width =6truecm}}
\put(-0.2,3){\makebox(0,0)[c]{$\infty$}}
\put(6.2,3){\makebox(0,0)[c]{$0$}}
\put(3,-.2){\makebox(0,0)[c]{$-1$}}
\put(3,6.2){\makebox(0,0)[c]{$1$}}
\put(2.5,3.3){\makebox(0,0)[c]{$L$}}
\put(3.5,3.3){\makebox(0,0)[c]{$R$}}
\end{picture}
\caption{The action of $R$ and $L$ on the Farey tessellation}\label{fig.RLaction}
\end{center}
\end{figure}

\section{$SL(2,\ZZ)$ and
the
Farey tessellation}
\label{sec.farey}
Recall that the mapping-class group of
the once-punctured torus $T=(\RR^2-\ZZ^2)/\ZZ^2$,
is identified with $SL(2,\ZZ)$.
Thus we may assume the pseudo-Anosov
homeomorphism $\varphi$ is a \lq linear' homeomorphism
determined by a matrix
$\begin{pmatrix}
a&b\\c&d\end{pmatrix}
\in SL(2,\ZZ)$
with $|a+d|>2$.

The homeomorphism $\varphi$ descends to an automorphism
of the orbifold $\OO$, denoted by the same symbol,
and its action $\varphi_*$
on the set of slopes of the (elliptic) generators
is given by the rule
\begin{equation}
\label{linear-fractional-transformation}
s\mapsto
\frac{c+ds}{a+bs}.
\end{equation}
We note that the right and left automorphisms $R$ and $L$
of $\pi_1(\OO)$ defined by (\ref{DefiningRL})
are induced by the automorphisms of $\OO$
corresponding to the following matrices,
which we represent by the same symbols:
\begin{equation}
\label{MatricesRL}
R=\begin{pmatrix}
         1&1\\
         0&1
        \end{pmatrix},
\quad
L=\begin{pmatrix}
         1&0\\
         1&1
        \end{pmatrix}.
\end{equation}

The rule (\ref{linear-fractional-transformation})
determines the isometric
action $\varphi_*$ on $\HH^2$
preserving the Farey tessellation.
Since  $|a+d|>2$,
$\varphi_*:\HH^2\to\HH^2$ is a hyperbolic translation,
and it has a unique attractive (resp. repulsive)
fixed point $\mu_+$ (resp. $\mu_-$)
on the ideal boundary $\RRR=\RR\cup\{\infty\}$.
Since $\mu_{\pm}$ are irrationals,
the oriented geodesic $\ell$ in $\HH^2$
running from $\mu_-$ to $\mu_+$
crosses infinitely many Farey triangles
$\cdots, \sigma_{-1},\sigma_{0},\sigma_{1},\sigma_{2},\cdots$.
This determines a bi-infinite word $\Omega=\prod f_n$
in the letters $\{R,L\}$
by the rule that $f_n$ is $R$ (resp. $L$)
if $\ell$ exits the Farey triangle $\sigma_n$
to the right (resp. left)
of where it enters.
Since $\varphi_*$ preserves the Farey tessellation,
there is a
unique (positive) integer $p$
such that
\begin{equation}
\label{Period}
\varphi_*(\sigma_n)=\sigma_{n+p}.
\end{equation}
Then $\Omega=(\prod_{n=1}^{p} f_n)^{\infty}$.

After conjugation, we may assume that
$\sigma_0=\langle 0,\infty,-1\rangle$ and
$\sigma_1=\langle 0,1,\infty\rangle$.
Then it follows
that $\varphi_*$
is equal to
$\prod_{n=1}^{p}(f_n)_*$
as isometries of $\HH^2$ preserving the Farey tessellation,
where $(f_n)_*$ is the isometry $R_*$ or $L_*$
induced by the matrix $R$ or $L$ in (\ref{MatricesRL})
according as the symbol $f_n$ is $R$ or $L$.
Thus it follows that $\varphi$ is equal to $\pm\prod_{n=1}^{p}
f_n$ as an element of $SL(2,\ZZ)$.
Since the word $\Omega$ contains both $R$ and $L$,
we may assume
$f_1=R$ and
$f_{0}=f_{p}=L$
after a shift of indices.
This implies the well-known fact that
$\varphi$ is conjugate in $SL(2,\ZZ)$ to
\begin{equation}
\label{RLdecomposition}
\pm R^{a_1}L^{b_1}R^{a_2}L^{b_2}\cdots R^{a_k}L^{b_k}
\end{equation}
for some $k\ge 1$ and positive integers $a_i$ and $b_i$,
where $p=\sum_{i=1}^k(a_i+b_i)$.
This word, up to cyclic permutation,
is uniquely determined by
the conjugacy class of $\varphi$.
We summarize our convention.

\begin{convention}
\label{convention.Farey}
For the pseudo-Anosov homeomorphism $\varphi$ of $T$,
$\{\sigma_n\}_{n\in\ZZ}$
denotes the bi-infinite sequence
of Farey triangles and
$\{f_n\}_{n\in\ZZ}$ denotes the bi-infinite sequence
of the letters $R$ and $L$ defined in the above.
We assume the following conditions are satisfied.
\rm
\begin{enumerate}
\item
$\sigma_0=\langle 0,\infty,-1\rangle$ and
$\sigma_1=\langle 0,1,\infty\rangle$.
\item
$f_{0}=f_{p}=L$ and $f_1=R$.
\item
$\varphi=\pm R^{a_1}L^{b_1}R^{a_2}L^{b_2}\cdots R^{a_k}L^{b_k}$,
where $k$, $a_i$, $b_i$ are positive integers
such that $p=\sum_{i=1}^k(a_i+b_i)$.
\end{enumerate}
\end{convention}

\section{Type-preserving representations of $\pi_1(\OO)$}
\label{sec.representation}
A $\PSL(2,\CC)$-representation of
$\pi_1(T)$ (resp. $\pi_1(\OO)$) is said to be
\textit{type-preserving}
if it is irreducible
(equivalently,
it does not have a global
fixed point
in $\bar\HH^3$)
and sends the distinguished
element $D$ to a parabolic transformation.
It is well-known that
every type-preserving representation of $\pi_1(T)$
uniquely extends to a
type-preserving representation of $\pi_1(\OO)$
(see e.g. \cite[Section 2]{Jorgensen}).
Throughout this paper,
we always assume that a type-preserving representation
$\rho:\pi_1(\OO)\to \PSL(2,\CC)$
is normalized so that
\begin{equation}
\label{NormalizingD}
\rho(D)=
\begin{pmatrix} 1 & 1\\
0 &1\end{pmatrix}.
\end{equation}
Thus $\rho(D)$ is a parabolic transformation fixing
the ideal point $\infty$.

For an elliptic generator $P$,
we have $\rho(P)(\infty)\ne \infty$
if $\rho$ is faithful,
because
$\rho(P)$ fixes $\infty$
if and only if $\rho(DP)$
is an elliptic transformation of order $2$
(cf. \cite[Proposition 2.4.4]{ASWY}), whereas
$DP\in\pi_1(T)$
is a hyperbolic element
when we identify $\pi_1(T)$ with a Fuchsian group.

Let $\sigma$ be a Farey triangle
and $\{P_j\}_{j\in\ZZ}$
the sequence of elliptic generators associated with $\sigma$.
Let $\rho:\pi_1(\OO)\to \PSL(2,\CC)$
be a type-preserving representation and
assume that none of
the $\rho(P_j)$ fix
$\infty$.
Then $\{\rho(P_j)(\infty)\}_{j\in\ZZ}$
is a sequence of points in $\CC$
which is invariant by the Euclidean
translation $z\mapsto z+1$,
because
\begin{equation}
\label{trasition-equation}
\rho(P_{j+3})(\infty)=\rho(DP_{j}D^{-1})(\infty)
=\rho(DP_{j})(\infty)=\rho(P_{j})(\infty)+1.
\end{equation}
We denote by $\LL(\rho,\sigma)$
the periodic (possibly singular) piecewise-straight line in $\CC$
obtained by joining the points $\{\rho(P_j)(\infty)\}_{j\in\ZZ}$
successively.
If $\rho$ is a
faithful discrete
$\PSL(2,\RR)$-representation
of $\pi_1(\OO)$,
then $\LL(\rho,\sigma)$ gives a triangulation of the real line,
and the hyperbolic plane lying above the real line
is identified with the universal cover $\tilde\OO$.
Moreover, the vertical geodesic joining $\infty$ and
the vertex $\rho(P_j)(\infty)$ corresponds to the edge
$e_j$ introduced in Section \ref{sec.orbifold}.
In general,
(the conjugacy class of) the
representation $\rho$
can be
recovered from $\LL(\rho,\sigma)$
(cf. \cite[Section 2.4]{ASWY}),
and it plays a key role
in the construction of the triangulation $\Delta(\varphi)$
induced by the canonical decomposition of the punctured-torus bundle
$\Tbundle$.

Since the hyper-elliptic involution of $T$
generates
the center of the mapping-class group of $T$,
it extends to a fiber-preserving involution
of $\Tbundle$.
Moreover, it is realized by an isometric involution
$\iota$ of the
hyperbolic manifold $\Tbundle$.
The quotient orbifold, $\Obundle$,
is a bundle over $S^1$
with fiber $\OO$
and admits a complete hyperbolic structure
of finite volume.
Let $\hat\Gamma\cong \pi_1(\Obundle)$ be
the
Kleinian group uniformizing the hyperbolic orbifold,
and let $\rho_{\CCC}:\pi_1(\Obundle) \to \hat\Gamma
\subset \PSL(2,\CC)$
be the holonomy representation.
(As in \cite{Cannon-Dicks2},
 the notation $\rho_{\CCC}$
records the fact that the limit set of
$\hat\Gamma$ is the whole Riemann sphere $\CCC$.)
The bundle structure
gives an
exact sequence
\begin{equation}
\label{short-exact-sequence}
\CD
1 @>>> \pi_1(\OO)
@>>> \pi_1(\Obundle)
@>>> \ZZ
@>>> 1,
\endCD
\end{equation}
and the restriction of $\rho_{\CCC}$ to $\pi_1(\OO)$
is type-preserving.

The orbifold $\Obundle$ can be compactified with
a single cusp with associated group $\ZZ^2$.
Deleting a small open neighborhood of the $\ZZ^2$-cusp
leaves a compact orbifold with the same orbifold
fundamental group as
$\Obundle$ and with one boundary component;  this boundary is
a torus which we consider fixed
and we call it
the \textit{peripheral torus},
and by abuse of notation we denote it  by $\partial\Obundle$.
This terminology lifts from $\Obundle$ to $\Tbundle$.

With respect to lifting the $\ZZ^2$-cusp to $\infty$, the
fundamental group of the peripheral torus $\partial\Obundle$
is generated by the distinguished element $D$
and an element, $D^{\dagger}$,
where $D^{\dagger}$ projects to the generator $1$ of $\ZZ$
in the exact sequence (\ref{short-exact-sequence}).
Since
$\rho_{\CCC}(D)=
\begin{pmatrix} 1 & 1\\
0 &1\end{pmatrix}$
by the normalization,
we have
$\rho_{\CCC}(D^{\dagger})=
\begin{pmatrix} 1 & \lambda\\
0 &1\end{pmatrix}$
for some non-real complex number $\lambda$.
Under a suitable orientation convention,
we may assume $\Im(\lambda)>0$.
Thus the stabilizers, $\Gamma_{\infty}$ and $\hat\Gamma_{\infty}$, of
the ideal point $\infty$ with respect to the actions of
$\Gamma$ and $\hat\Gamma$, respectively, are given as follows.
\begin{align}
\label{T-boundary}
\Gamma_{\infty}
&=
\left\langle
\begin{pmatrix} 1 & 2\\
0 &1\end{pmatrix},
\begin{pmatrix} 1 & \lambda\\
0 &1\end{pmatrix}
\right\rangle
\cong
\pi_1(\partial\Tbundle)\\
\label{O-boundary}
\hat\Gamma_{\infty}
&=
\left\langle
\begin{pmatrix} 1 & 1\\
0 &1\end{pmatrix},
\begin{pmatrix} 1 & \lambda\\
0 &1\end{pmatrix}
\right\rangle
\cong
\pi_1(\partial\Obundle).
\end{align}
We may choose our peripheral torus
$\partial \Tbundle$ so that
its preimage in hyperbolic three-space
is a family of horospheres.
The horosphere at $\infty$
is acted on by $\Gamma_{\infty}$, and
can be identified with $\CC$ by projection from $\infty$,
and thus $\partial \Tbundle$ is identified with
the quotient space $\CC/\Gamma_{\infty}$.
Similar identifications hold for $\partial\Obundle$.

\section{The canonical decomposition of $\Tbundle$}
\label{sec.cusptriangulation}
Recall that each Farey triangle $\sigma$
determines a (topological) ideal triangulation $\trg(\sigma)$
of the punctured torus $T$.
Moreover, if $\sigma$ and $\sigma'$ are mutually adjacent Farey triangles,
then $\trg(\sigma')$ is obtained from $\trg(\sigma)$
by a \lq\lq diagonal exchange'',
i.e.,
by deleting
any one of the three edges and then inserting a new edge in the unique possible way.
As illustrated in Figure \ref{fig.diagonal-exchange},
$\trg(\sigma)$ and $\trg(\sigma')$
can be regarded as the bottom and top faces of an immersed
topological ideal tetrahedron
(with two
pairs of edges
identified)
in $T\times \RR$.
We denote this immersed topological ideal tetrahedron in
$T\times\RR$ by $\trg(\sigma,\sigma')$.

\begin{figure}[t!]
\begin{center}
\setlength{\unitlength}{1truecm}
\begin{picture}(9,3.4)
\put(0,.7){\epsfig{file= 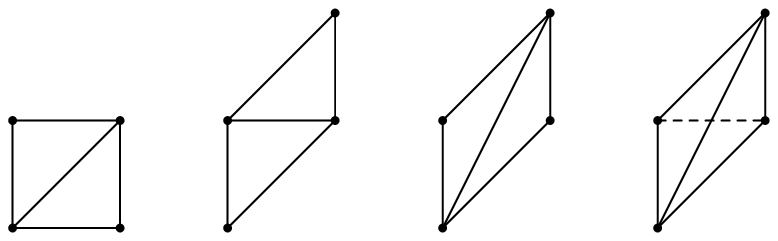, width=9truecm}}
\put(0.8,.25){\makebox(0,0)[c]{$\trg(\sigma)$}}
\put(2.2,1.5){\makebox(0,0)[c]{$=$}}
\put(3.2,.25){\makebox(0,0)[c]{$\trg(\sigma)$}}
 \put(4.2,2){\vector(1,0){.6}}
\put(5.5,.25){\makebox(0,0)[c]{$\trg(\sigma')$}}
\put(7.9,.25){\makebox(0,0)[c]{$\trg(\sigma,\sigma')$}}
\end{picture}
\caption{$\trg(\sigma)$ and $\trg(\sigma')$ form
the immersed topological ideal tetrahedron $\trg(\sigma,\sigma')$}\label{fig.diagonal-exchange}
\end{center}
\end{figure}

The immersed topological ideal tetrahedra
$\{\trg(\sigma_{n},\sigma_{n+1})\}_{n\in\ZZ}$
can be
stacked
up
to form a topological ideal triangulation of $T\times \RR$.
Since $\varphi_*(\sigma_n)=\sigma_{n+p}$ for every integer $n$,
we may assume that the covering transformation
$(x,t)\mapsto (\varphi(x), t+1)$, of the infinite-cyclic
covering $T\times \RR$ of $\Tbundle$,
sends  $\trg(\sigma_{n},\sigma_{n+1})$
to $\trg(\sigma_{n+p},\sigma_{n+p+1})$
and hence it preserves the topological ideal triangulation
of $T\times\RR$.
Thus
there is an induced
topological ideal triangulation of $\Tbundle$
consisting of $p$ ideal tetrahedra and
$2p$ ideal triangles and $p$ ideal edges.
The following theorem was found by J\o rgensen \cite{Jorgensen}
(cf. \cite{Floyd-Hatcher})
and
rigorous treatments were given
in
\cite{Akiyoshi, Lackenby, Gueritaud1,
Gueritaud2, Jorgensen, Parker}.

\begin{theorem}
\label{theorem.Jorgensen}
The topological ideal triangulation
of $\Tbundle$ described in the above
is
homeomorphic
to the canonical
decomposition of the complete hyperbolic manifold $\Tbundle$.
\end{theorem}

In particular, the triangulation of the chosen peripheral torus $\partial \Tbundle$
induced by the canonical tetrahedral decomposition of $\Tbundle$
is combinatorially isomorphic to the triangulation
induced by the above combinatorial tetrahedral decomposition.
We now describe this combinatorial triangulation
following \cite{Gueritaud1}.
Since $\trg(\sigma_{n})$ is an ideal triangulation
of a fiber surface $T$ consisting of two ideal triangles,
it induces a triangulation, $C(\sigma_n)$, of
some chosen
\textit{peripheral circle}
$\partial T$ in $T$.
The triangulation $C(\sigma_n)$ consists of $6$ edges,
which correspond to the $6$
ideal vertices of the two
ideal triangles.
The region in $\partial T\times \RR$ bounded by $C(\sigma_n)$ and
$C(\sigma_{n+1})$ consists of $4$ triangles,
which correspond to the $4$ ideal vertices
of the tetrahedron $\trg(\sigma_{n},\sigma_{n+1})$
(see Figure \ref{fig.trucated-tetrahedron}).

\begin{figure}[t!]
\begin{center}
\setlength{\unitlength}{1truecm}
\begin{picture}(10.5,7.9)
\put(0,0){\epsfig{file= 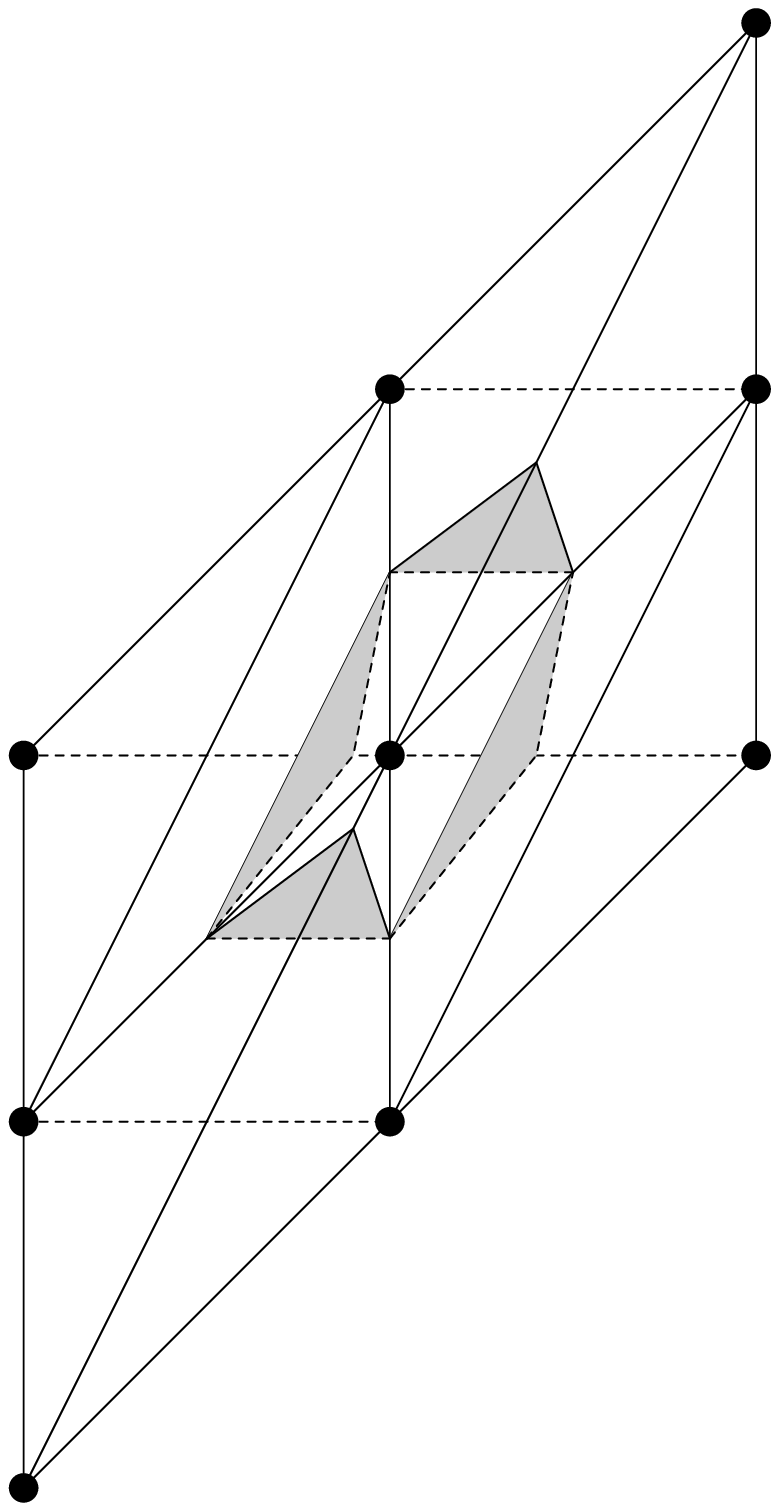, width=4truecm}}
\put(4.5,1.6){\epsfig{file= 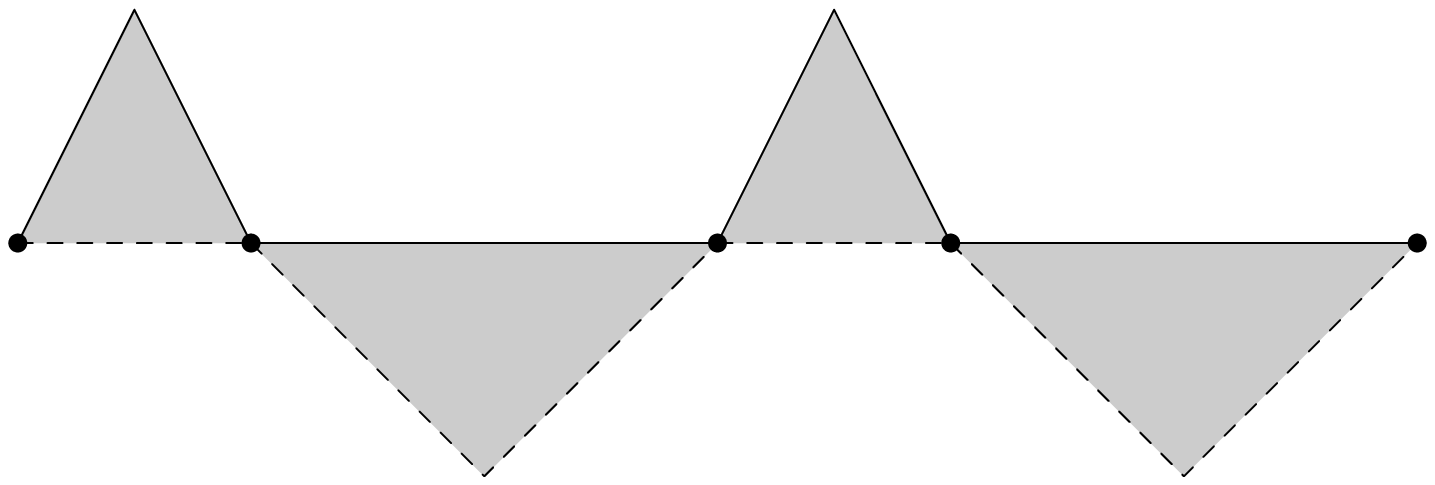, width=6truecm}}
\put(1.9,3.5){\vector(1,0){3}}
\put(1,3.2){\makebox(0,0)[c]{$a$}}
\put(2.2,3){\makebox(0,0)[c]{$b$}}
\put(3.1,4.8){\makebox(0,0)[c]{$c$}}
\put(1.85,5){\makebox(0,0)[c]{$d$}}
\put(4.6,2.3){\makebox(0,0)[c]{$a$}}
\put(5.15,2.9){\makebox(0,0)[c]{$\circlearrowright$}}
\put(5.5,2.3){\makebox(0,0)[c]{$b$}}
\put(6.5,2.3){\makebox(0,0)[c]{$\circlearrowleft$}}
\put(7.5,2.3){\makebox(0,0)[c]{$c$}}
\put(7.97,2.9){\makebox(0,0)[c]{$\circlearrowright$}}
\put(8.3,2.3){\makebox(0,0)[c]{$d$}}
\put(9.37,2.3){\makebox(0,0)[c]{$\circlearrowleft$}}
\put(10.3,2.3){\makebox(0,0)[c]{$a$}}
\put(5.5,3.8){\makebox(0,0)[c]{upward apex}}
\put(6.7,1.4){\makebox(0,0)[c]{downward apex}}
\end{picture}
\caption{The developed image of the triangles corresponding to the
$4$ ideal vertices of the ideal tetrahedron}\label{fig.trucated-tetrahedron}
\end{center}
\end{figure}

Since the family $\{\trg(\sigma_{n})\}_{n\in\ZZ}$
forms the $2$-skeleton of the
ideal triangulation of $T\times\RR$,
invariant by the covering transformation
$(x,t)\mapsto (\varphi(x), t+1)$,
the family $\{C(\sigma_n)\}_{n\in\ZZ}$
forms the $1$-skeleton of a triangulation
of $\partial T\times \RR$
invariant by the covering transformation.
It descends to a triangulation of
the peripheral torus
$\partial \Tbundle$.
This is the triangulation induced by the topological ideal triangulation
of $\Tbundle$.
Let $\Delta^*(\varphi)$ be the lift of
this triangulation of $\partial\Tbundle$ to
its universal cover $\widetilde{\partial\Tbundle}$.
Since $\widetilde{\partial\Tbundle}$ is identified
with the universal cover $\widetilde{\partial\Obundle}$ of
$\partial\Obundle$ and since
the above triangulation of $\partial\Tbundle$
is invariant by the hyper-elliptic involution,
$\Delta^*(\varphi)$ gives a triangulation of
$\widetilde{\partial\Obundle}$ invariant by
$\pi_1(\partial\Obundle)$.
We identify $\widetilde{\partial\Obundle}$ with the complex plane $\CC$,
where the action of the generators $D$ and $D^{\dagger}$
of $\pi_1(\partial\Obundle)$ are given by the
following
formulas.
\begin{equation}
D(z)=z+1, \quad D^{\dagger}(z)=z+\lambda.
\end{equation}
Here $\lambda$ is the complex number in (\ref{O-boundary}).
Then the inverse image of $C(\sigma_n)$
in $\CC$
is a bi-infinite, horizontal, piecewise-straight
line, $\LL(\sigma_n)$,
which is invariant by $D$.
The translation $D$ shifts each vertex and edge of $\LL(\sigma_n)$
to the right by $3$ units.
The piecewise-straight line $\LL(\sigma_{n+1})$ lies above $\LL(\sigma_n)$, and
the $1$-skeleton of $\Delta^*(\varphi)$ is obtained by
stacking up these bi-infinite piecewise-straight lines.
Thus the set $\{\LL(\sigma_n)\}_{n\in\ZZ}$
gives a layered structure of
$\Delta^*(\varphi)$ in the sense of Definition \ref{layered
complex} (1)
below.
Moreover,  $(\Delta^*(\varphi),\{\LL(\sigma_n)\}_{n\in\ZZ})$
can be regarded as
a layered $\pi_1(\partial\Obundle)$-simplicial complex
in the sense of Definition \ref{layered complex} (3) below.
See Figure \ref{fig.cusp-triangulation}, where
the vertices are \lq\lq opened up''
in order to emphasize the layered structure.

\begin{definition}
\label{layered complex}
\rm
(1) By a \textit{layered structure} of a
$2$-dimensional simplicial complex $K$ with underlying space $\CC$,
we mean a  family of $1$-dimensional subcomplexes
$\{\LL_n\}_{n\in\ZZ}$ indexed by $\ZZ$,
satisfying the following conditions.
\begin{enumerate}[(i)]
\item
Each $\LL_n$ gives a triangulation of
a subspace homeomorphic to the real line $\RR$,
and $\cup_{n\in\ZZ}\LL_n$ forms the $1$-skeleton of $K$.
\item
For each pair $\LL_m$ and $\LL_n$ with $m\ne n$,
the union of $\LL_m$ and $\LL_n$ cuts off a region in $\CC$
homomorphic to a (possibly-pinched) infinite strip.
\item
$\{\LL_n\}_{n\in\ZZ}$ lie in $\CC$ in the order of the index
from bottom to top.
\end{enumerate}
We call the pair $(K,\{\LL_n\}_{n\in\ZZ})$
a \textit{layered simplicial complex}.
It is often abbreviated to $(K,\{\LL_n\})$.

(2) By an \textit{isomorphism} between two
layered simplicial complexes
$(K,\{\LL_n\})$ and $(K',\{\LL_n'\})$
we mean a simplicial isomorphism
from $K$ to $K'$
which maps $\LL_n$ to $\LL_{n+d}$,
where $d$ is an integer independent of $n$.

(3) For a group $G$, a
\textit{layered $G$-simplicial complex}
is a layered simplicial complex $(K,\{\LL_n\})$
together with an action by $G$
where each element of $G$ acts as
an
automorphism of
$(K,\{\LL_n\})$.
An \textit{isomorphism} between
layered $G$-simplicial complexes
is defined to be a $G$-equivariant
isomorphism between the layered simplicial complexes.
\end{definition}

\begin{figure}[t!]
\begin{center}
\setlength{\unitlength}{1truecm}
\begin{picture}(12,10.8)
\put(2.9,1.75){\makebox(0,0)[c]{$\textstyle\sigma_{\scriptscriptstyle{\!-3}}$}}
\put(2.87,2.41){\makebox(0,0)[c]{$\textstyle\sigma_{\scriptscriptstyle{\!-2}}$}}
\put(2.55,2.75){\makebox(0,0)[c]{$\textstyle\sigma_{\scriptscriptstyle{\!-1}}$}}
\put(2.2,3.2){\makebox(0,0)[c]{$\textstyle\sigma_{\scriptscriptstyle{\!0}}$}}
\put(1.6,3.6){\makebox(0,0)[c]{$\textstyle\sigma_{\scriptscriptstyle{\!1}}$}}
\put(1.7,4.55){\makebox(0,0)[c]{$\textstyle\sigma_{\scriptscriptstyle{\!2}}$}}
\put(1.7,5.2){\makebox(0,0)[c]{$\textstyle\sigma_{\scriptscriptstyle{\!3}}$}}
\put(1.05,5.55){\makebox(0,0)[c]{$\textstyle\sigma_{\scriptscriptstyle{\!4}}$}}
\put(1.2,6.35){\makebox(0,0)[c]{$\textstyle\sigma_{\scriptscriptstyle{\!5}}$}}
\put(1.15,7.0){\makebox(0,0)[c]{$\textstyle\sigma_{\scriptscriptstyle{\!6}}$}}
\put(1.75,7.2){\makebox(0,0)[c]{$\textstyle\sigma_{\scriptscriptstyle{\!7}}$}}
\put(1.75,7.815){\makebox(0,0)[c]{$\textstyle\sigma_{\scriptscriptstyle{\!8}}$}}
\put(1.45,8.155){\makebox(0,0)[c]{$\textstyle\sigma_{\scriptscriptstyle{\!9}}$}}
\put(1.2,8.63){\makebox(0,0)[c]{$\textstyle\sigma_{\scriptscriptstyle{\!10}}$}}
\put(6.75,1.65){\makebox(0,0)[c]{$\scriptstyle\circlearrowright$}}
\put(8.5,1.15){\makebox(0,0)[c]{$\scriptstyle\circlearrowleft$}}
\put(0,.1){\epsfig{file=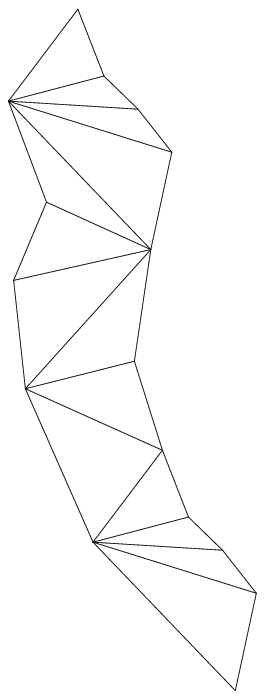, width =3.6truecm}}
\put(3.9,0){\epsfig{file=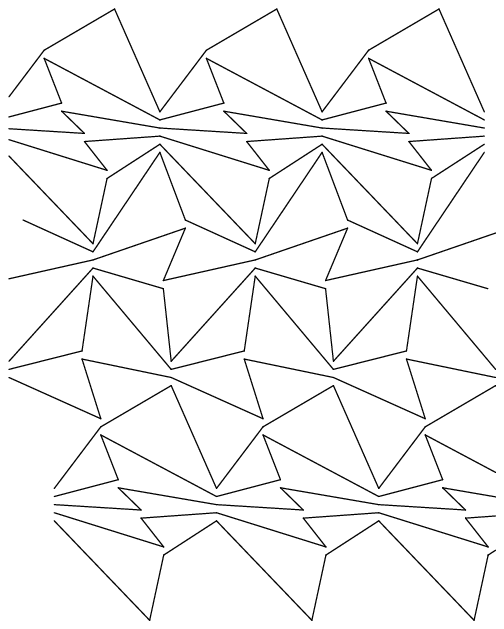, width =8.1truecm}}
\end{picture}
\caption{The layered  simplicial complex
$(\Delta^*(\varphi),\{\LL(\sigma_n)\})$}\label{fig.cusp-triangulation}
\end{center}
\end{figure}

By Theorem \ref{theorem.Jorgensen},
$\Delta^*(\varphi)$ gives a combinatorial model
of $\Delta(\varphi)$.
To describe the correspondence,
consider the $\pi_1(\OO)$-invariant ideal triangulation
of $\HH^3$, obtained as the lift of the canonical decomposition
of the hyperbolic manifold $\Tbundle$.
Then $\Delta(\varphi)$ is obtained as its intersection
with the horosphere $\CC\times\{t_0\}$,
where $t_0$ is a sufficiently large positive real number.
Thus a vertex of $\Delta(\varphi)$ corresponds to a
vertical edge (i.e., an edge emanating from $\infty$)
of the lifted canonical decomposition, and vice versa.
Similarly, an edge of $\Delta(\varphi)$
corresponds to a vertical face
(i.e., a face having $\infty$ as a vertex)
of the lifted canonical decomposition, and vice versa.
Recall that we identify $\Delta(\varphi)$ with its
projection to the complex plane $\CC\subset\partial\HH^3$.
Then a vertex of $\Delta(\varphi)$
is equal to the endpoint of the corresponding vertical edge
of the lifted canonical decomposition,
and an edge of $\Delta(\varphi)$
is equal to the line segment
joining two vertices
obtained as the projection
of the corresponding face
of the lifted canonical decomposition.

Since the layer $\LL(\sigma_n)$ of $\Delta^*(\varphi)$
arises from the ideal triangulation $\trg(\sigma_n)$,
it follows that
the vertex set of its image in $\Delta(\varphi)$
is equal to $\{\rho_{\CCC}(P_j^{(n)})(\infty)\}_{j\in\ZZ}$,
where $\{P_j^{(n)}\}_{j\in\ZZ}$ is the sequence of
elliptic generators associated with $\sigma_n$,
and
$\rho_{\CCC}:\pi_1(\OO)\to \PSL(2,\CC)$
is the type-preserving representation obtained from
the holonomy representation of
the complete hyperbolic orbifold $\Obundle$.
Thus the layer $\LL(\sigma_n)$ of $\Delta^*(\varphi)$
corresponds to the periodic piecewise-straight line
$\LL(\rho_{\CCC},\sigma_n)$.
Hence we obtain the following theorem.

\begin{theorem}
\label{thm.canonical-decomposition}
The family $\{\LL(\rho_{\CCC},\sigma_n)\}_{n\in\ZZ}$ forms a layered structure of
$\Delta(\varphi)$ invariant by the action of $\pi_1(\partial\Obundle)$, and
the layered
$\pi_1(\partial\Obundle)$-simplicial complexes
$(\Delta(\varphi), \{\LL(\rho_{\CCC},\sigma_n)\})$
and  $(\Delta^*(\varphi),\{\LL(\sigma_n)\})$ are isomorphic.
In particular, $\Delta(\varphi)$ is described as follows.
\begin{enumerate}[\normalfont (1)]
\item
The vertex set of $\Delta(\varphi)$ consists of the points
\[
\rho_{\CCC}(P_j^{(n)})(\infty).
\]
\item
The edge set of $\Delta(\varphi)$ consists of
\[
\langle \rho_{\CCC}(P_j^{(n)})(\infty),\rho_{\CCC}(P_{j+1}^{(n)})(\infty) \rangle.
\]
\item
The face set of $\Delta(\varphi)$ consists of the convex hulls of
\[
\{\rho_{\CCC}(P_{j}^{(n)})(\infty),\rho_{\CCC}(P_{j+1}^{(n)})(\infty),
\rho_{\CCC}(P_{j+2}^{(n)})(\infty)\},
\]
such that $(P_{j}^{(n)},P_{j+2}^{(n)})$ is a pair of
successive elements in the sequence of elliptic generators
associated with $\sigma_{n-1}$ or
$\sigma_{n+1}$.
\end{enumerate}
Here
$\{P_j^{(n)}\}_{j\in\ZZ}$ is the sequence of
elliptic generators associated with $\sigma_n$.

\end{theorem}

The above argument also implies the following combinatorial description
of the topological model $\Delta^*(\varphi)$ of $\Delta(\varphi)$.

\begin{proposition}
\label{prop.model-Delta}
{\normalfont(1)}
The vertex set of $\Delta^*(\varphi)$
is identified with the set of the elliptic generators $P$
such that $s(P)$ is a vertex of $\sigma_n$ for some $n\in\ZZ$.

{\normalfont(2)}
The edge set of $\Delta^*(\varphi)$
is identified with the set of pairs $(P,Q)$
of elliptic generators,
such that $P$ and $Q$ are two successive elements
in the sequence of elliptic generators associated with
$\sigma_n$ for some $n\in\ZZ$.

{\normalfont(3)}
The face set of $\Delta^*(\varphi)$
is identified with the set of triples $(P,Q,R)$
of elliptic generators,
such that the following hold for some $n\in\ZZ$.
\begin{enumerate}[\normalfont (i)]
\item
$(P,Q,R)$ is an elliptic generator triple
associated with $\sigma_n$.
\item
$P$ and $R$ are two successive elements
of the sequence of elliptic generators associated with
$\sigma_{n-1}$ or $\sigma_{n+1}$.
\end{enumerate}

{\normalfont(4)} The layer $\LL(\sigma_n)$ corresponds to the union of
the bi-infinite family of edges
$\{\langle P_j^{(n)},P_{j+1}^{(n)}\rangle\}_{j\in\ZZ}$,
where $\{P_j^{(n)}\}_{j\in\ZZ}$ is the sequence of
elliptic generators associated with $\sigma_n$
\end{proposition}

\section{The Cannon-Thurston map}
\label{sec.Cannon-Thurston}

In this section,
we recall the combinatorial description of the
Cannon-Thurston map following \cite{Cannon-Dicks1, Cannon-Dicks2}.
Recall that we have identified $\pi_1(\OO)$ with a Fuchsian group
by choosing a complete hyperbolic metric of $\OO$
of finite area.
Let
$\rho_{\RRR}:\pi_1(\OO)\to\PSL(2,\RR)\subset \PSL(2,\CC)$
be the holonomy representation inducing the identification.
Then the \textit{limit set},
$\Lambda(\rho_{\RRR})$,
of the group $\rho_{\RRR}(\pi_1(\OO))$
is equal to the round circle $\RRR$.

We recall a combinatorial description of the
$\pi_1(\OO)$-space $\Lambda(\rho_{\RRR})$
following \cite{Alperin-Dicks-Porti} (cf. \cite{Floyd}).
Let $\End$ be the space of the ends
of the group $\pi_1(\OO)$.
Here an \textit{end} of $\pi_1(\OO)$
is an infinite reduced word in $A$, $B$, $C$,
i.e., an infinite sequence
$X_1X_2\cdots X_n\cdots$
such that $X_n\in\{A,B,C\}$  and $X_n\ne X_{n+1}$
for every $n\in\NN$.
For a finite reduced
word $F$ in $A$, $B$, $C$,
let $[F]$ be the subset of $\End$
consisting of those infinite words
which has $F$ as an initial segment.
The space $\End$ is endowed with the topology
such that
the family $\{[F]\}$,
where $F$ runs over all finite reduced words,
forms a base of the topology.
Then $\End$ is a Cantor set,
i.e., a compact, Hausdorff, totally disconnected,
topological space.
Note that the left multiplication of $\pi_1(\OO)$ induces
a left action of $\pi_1(\OO)$ on $\End$.

For an infinite reduced word $W$,
let $W_n$ be the initial $n$-th segment of $W$.
Then $\lim \rho_{\RRR}(W_n)(\infty)$ exists in $\RRR$ for every
infinite reduced word $W\in\End$,
and the map $\End\to \RRR=\Lambda(\rho_{\RRR})$
defined by $W\mapsto \lim \rho_{\RRR}(W_n)(\infty)$
is a continuous $\pi_1(\OO)$-equivariant surjective map.
Moreover, if $\sim$ denotes the equivalence relation
on $\End$ generated by the relation
\begin{equation}
\label{fuchsian-equivalence-relation}
WD^{\infty}\sim WD^{-\infty}
\end{equation}
where $W$
runs over elements of $\pi_1(\OO)$,
then the above continuous map induces a
$\pi_1(\OO)$-equivariant homeomorphism
from $\End/{\sim}$ to $\Lambda(\rho_{\RRR})=\RRR$.
It should be noted that
the equivalence relation $\sim$
is independent of the choice of a
complete hyperbolic structure of $\OO$.

Recall the type-preserving representation
$\rho_{\CCC}:\pi_1(\OO)\to\PSL(2,\CC)$
associated with the complete hyperbolic orbifold
$\Obundle$. Then the limit set
$\Lambda(\rho_{\CCC})$ of the Kleinian group
$\rho_{\CCC}(\pi_1(\OO))$ is equal to the whole Riemann
sphere
$\CCC$. It was proved by McMullen \cite{McMullen}
that
for every infinite word $W\in\End$
the sequence $\rho_{\CCC}(W_n)(\infty)$ converges in $\CCC$
and that the map
$\End\to \CCC$
defined by $W\mapsto \lim \rho_{\CCC}(W_n)(\infty)$
is a continuous $\pi_1(\OO)$-equivariant surjective map,
which in turn
induces a continuous $\pi_1(\OO)$-equivariant surjective map
\[
\CT:\RRR=\Lambda(\rho_{\RRR})\to\Lambda(\rho_{\CCC})=\CCC.
\]
This is called the \textit{Cannon-Thurston map}
associated to the punctured-torus bundle $\Tbundle$.

Following \cite{Cannon-Dicks1, Cannon-Dicks2}, we now recall the
combinatorial description of the Cannon-Thurston map $\CT$
established by Bowditch \cite{Bowditch}. Since we have fixed a
complete hyperbolic structure on $\OO$, we can identify the
universal cover of $\OO$ with the upper half-space
$\CC_{+}:=\{z\in\CC \mid \Im z>0\}$. Thus we obtain the following
tower of (topological) coverings:
\begin{equation}
\label{covering-tower}
\CC_{+}\to (\RR^2-\ZZ^2) \to T\to \OO.
\end{equation}
Recall the attractive fixed point, $\mu_+$, of the linear fractional
action of $\varphi_*$ on $\bar\CC_{+}$ defined by
(\ref{linear-fractional-transformation}), i.e., the slope of the
expanding eigenspace of the linear map $\varphi\in SL(2,\ZZ)$.
Consider the foliation of $T$ determined by the foliation of
$\RR^2-\ZZ^2$ by lines of slope $\mu_+$. This foliation descends to
a foliation of $T$, which is the stable foliation of the
pseudo-Anosov homeomorphism $\varphi$, and it lifts to a foliation
of $\CC_{+}$ whose leaves are homeomorphic to the open interval.
Each end of each leaf has a well-defined endpoint on $\RRR$, and the
closure of each leaf, which we call a
\textit{
closed-up
leaf}, is homeomorphic to the closed interval.
Now let $\Parabolic$ be the set of the parabolic fixed points
of the Fuchsian group $\rho_{\RRR}(\pi_1(\OO))$.
Then two closed-up leaves have a common point if
and only if they share a point, say $p$, in $\Parabolic$.
Moreover,
either of these closed-up leaves is translated to the
other by a parabolic element of
$\rho_{\RRR}(\pi_1(\OO))$
with parabolic fixed point $p$.
For each $p\in\Parabolic$,
the union of the closed-up leaves with an endpoint $p$
is called the \textit{spider} with \textit{head} $p$.
Each of the closed-up leaves in the spider
is called a \textit{leg}, and the endpoint of a leg
different from the head is called a \textit{foot}.
It should be noted that the image of a spider in $\RR^2-\ZZ^2$
consists of a puncture and the pair
of rays of slope $\mu_+$ emanating from the puncture.
We thus obtain a partition, $\Partition_+$,
of $\bar\CC_+=\CC_+\cup\RRR$ into
(i) spiders,
(ii) closed-up leaves disjoint from $\Parabolic$, and
(iii) one-point sets in $\RR$ disjoint from the endpoints
of the closed-up leaves.

By applying the complex conjugate to (\ref{covering-tower}), we obtain yet
another tower
of (topological) coverings:
\[
\CC_{-}\to (\RR^2-\ZZ^2) \to T\to \OO,
\]
where $\CC_{-}:=\{z\in\CC \mid \Im z<0\}$ is the lower half-space.
Starting from the foliation of $\RR^2-\ZZ^2$ by lines of slope
$\mu_-$, the repulsive fixed point of $\varphi_*$, we obtain a
similar partition, $\Partition_-$, of $\bar\CC_-=\CC_-\cup\RRR$. A
piece in $\Partition_+$ and that of $\Partition_-$, which are not
one-point sets, intersect if and only if they are spiders with a
common head, say $p\in\Parabolic$. Their union is called the
\textit{double spider} with head $p$. It should be noted that the
image of a double spider in $\RR^2-\ZZ^2$ consists of a puncture and
the two pairs of rays of slopes $\mu_+$ and $\mu_-$ emanating from
the puncture. So the feet of a double spider, contained in
$\Partition_+$ and $\Partition_-$ are arranged in $\RR$
alternately.
The \lq union' of $\Partition_+$ and $\Partition_-$
determines a partition, $\Partition$, of $\CCC$ into (i) double
spiders, (ii) closed-up leaves disjoint from $\Parabolic$, and (iii)
one-point sets in $\RR$ disjoint from the endpoints of the closed-up
leaves.

The continuous map
$\RRR/\Partition \to \CCC/\Partition$
between the quotient spaces
induced by the inclusion $\RRR\to\CCC$
is surjective,
because every piece in $\Partition$ intersects $\RRR$.
Since the actions of the Fuchsian group
$\rho_{\RRR}(\pi_1(\OO))$ on $\RRR$ and $\CCC$ respects the
partition
$\Partition$,
both quotient spaces inherit $\pi_1(\OO)$-actions,
and the above surjective continuous map
is $\pi_1(\OO)$-equivariant.

As is shown
in
\cite[Appendix]{Cannon-Dicks1},
the
Moore
decomposition theorem
implies that the quotient space $\CCC/\Partition$
is homeomorphic to the $2$-sphere.
We denote this topological $2$-sphere with the $\pi_1(\OO)$-action
by $\Sphere^2$.
Then we have the following tower of
continuous $\pi_1(\OO)$-equivariant surjective maps:
\[
\RRR=\Lambda(\rho_{\RRR})\to
\RRR/\Partition \to
\CCC/\Partition=\Sphere^2.
\]
Let $\CTM:\RRR=\Lambda(\rho_{\RRR})\to
\CCC/\Partition=\Sphere^2$
be the composition of these surjective maps
and call it the \textit{model Cannon-Thurston map}.
The following theorem was proved by Bowditch \cite{Bowditch}.

\begin{theorem}[Bowditch]
\label{thm.Bowditch}
The model Cannon-Thurston map $\CTM$
gives a combinatorial model of the Cannon-Thurston map.
Namely, there is a $\pi_1(\OO)$-equivariant homeomorphism
$\Phi:\Sphere^2\to\Lambda(\rho_{\CCC})=\CCC$
such that $\CT=\Phi\circ\CTM$.
\end{theorem}

By the above theorem and the description of the model
Cannon-Thurston map, it follows that
the inverse image
$\CT^{-1}(x)$ of a point $x\in\CCC$
consists of one, two, or countably infinite points.
If $|\CT^{-1}(x)|=2$,
then $\CT^{-1}(x)$ consists of the endpoints
of a closed-up leaf of one of the two foliations.
If $|\CT^{-1}(x)|=\infty$, then
$x$ is a parabolic fixed point of $\rho_{\hat\CC}(\pi_1(\OO))$,
i.e., there is an element $D'$ of $\pi_1(\OO)$ conjugate to $D$
such that $x$ is the fixed point of
the parabolic transformation $\rho_{\hat\CC}(D')$.
In this case
$\CT^{-1}(x)$ consists of
the fixed point of $\rho_{\hat\RR}(D')$
and the feet of the double spider
having the point as the head.

\section{The fractal tessellation $CW(\varphi)$}
\label{sec.fractal-tessellation}

In this section, we review the construction of the
fractal tessellation $CW(\varphi)$
introduced in
\cite{Cannon-Dicks2}.
Recall that the Cannon-Thurston map is obtained by shrinking each
lifted closed-up
leaf
of the stable and unstable  foliations to a point.
Roughly speaking,
the fractal tessellation is obtained by looking at the intersection
of the foliations
of the pseudo-Anosov homeomorphism
$\varphi:T\to T$,
and it reflects the way the Cannon-Thurston map
fills in $\CCC$.

Let $\spider$ be the double spider with head $\infty$,
the parabolic fixed point of $\rho_{\RRR}(D)$,
and let
$\{\leg_m\}_{m\in\ZZ}$ be the legs of $\spider$,
and let $w_m$ be the foot of $\leg_m$.
We assume that
the elements of
$\{w_m\}$ are located in $\RR$ in increasing order
and that $\leg_m$ is contained in
$\bar\CC_+$ or $\bar\CC_-$ according as
$m$ is even or odd.
It should be noted that $\rho_{\RRR}(D)$
maps $\leg_m$ to $\leg_{m+2}$
and that $\CT^{-1}(\infty)=
\{w_m\}_{m\in\ZZ}\cup\{\infty\}$.
Thus $\RRR-\CT^{-1}(\infty)$
is the disjoint union
$\sqcup_{m\in\ZZ}(w_m,w_{m+1})$
of open intervals.

Let $\bar\leg_m\subset\CCC$ be the complex conjugate of $\leg_m$,
and let
$\hat\partial_m$ be the image of $\bar\leg_m$
under the quotient map $\CCC\to \CCC/\Partition=\Sphere^2$.
The surjective map $\bar\leg_m\to\hat\partial_m$
is the quotient map that identifies the two endpoints
$\infty$ and $w_m$
of $\bar\leg_m$ (\cite[Proposition 7.8]{Cannon-Dicks2}).
Thus $\hat\partial_m$ is a simple closed curve in $\Sphere^2$
passing through $\inftym$,
where $\inftym=\Phi^{-1}(\infty)\in \Sphere^2$.
Hence
$\partial_m:=\hat\partial_m-\{\inftym\}$
is homeomorphic to the real line
and is properly embedded in
the model complex plane, $\CCM:=\Sphere^2-\{\inftym\}$.

We orient $\partial_m$ as follows.
We first orient
the loop $\ell_m\cup \bar\ell_m$ so that
it proceeds from $\infty$ through $\CC_-$ to $w_m$ and through
$\CC_+$ back to $\infty$.
Since $\hat\partial_m$ is equal to the image of
the loop $\ell_m\cup \bar\ell_m$ in $\Sphere^2$,
it inherits an orientation from $\ell_m\cup \bar\ell_m$
and so does $\partial_m$.
The following proposition is proved in
\cite[Section 7]{Cannon-Dicks2}.

\begin{proposition}
\label{prop.CW-decomposition}
{\normalfont(1)} $\partial_m\cap\partial_{m'}=\emptyset$
whenever $|m-m'|>1$.

{\normalfont(2)} $\partial_m\cap\partial_{m+1}$
is a discrete subset of $\partial_m$
which accumulates at $\inftym$ from both
directions.

{\normalfont(3)} The closure of each component of $\CCM-\cup_{m\in\ZZ}\partial_
{m}$
is homeomorphic to the disk.
\end{proposition}

By the first and the second assertions of
Proposition \ref{prop.CW-decomposition},
$\partial_{m-1}\cap\partial_{m}$ and
$\partial_{m}\cap\partial_{m+1}$
are mutually disjoint isolated subsets of the line
$\partial_{m}$, and their union
forms the vertex set of a CW-decomposition of
$\partial_{m}$.
By the last assertion of
Proposition \ref{prop.CW-decomposition},
the union of the $1$-dimensional cell complexes
$\{\partial_m\}_{m\in\ZZ}$
etches a $\langle D \rangle$-invariant
CW-decomposition of $\CCM$ satisfying the following conditions.
\begin{enumerate}[(i)]
\item
The vertex set is $\cup_{m\in\ZZ}(\partial_m\cap\partial_{m+1})$.
\item
The edge set is the set of the closures
of components of $\partial_m-(\partial_{m-1}\cup \partial_{m+1})$
where $m$ runs over $\ZZ$.
\item
The face set is the set of the closures of
the components of $\CCM-\cup_{m\in\ZZ}\partial_m$.
\end{enumerate}
We denote this CW-decomposition of $\CCM$ by $CW(\varphi)$.
Note that the $\pi_1(\OO)$-equivariant homeomorphism
$\Phi:\Sphere^2\to\Lambda(\rho_{\CCC})=\CCC$
restricts to a $\langle D \rangle$-equivariant homeomorphism
$\CCM\to \CC$.
We identify the $\langle D \rangle$-space $\CC$
with $\CCM$ through this homeomorphism.
Then the $\langle D \rangle$-invariant
CW-decomposition $CW(\varphi)$
of $\CCM$ determines a $\langle D \rangle$-invariant
CW-decomposition of $\CC$, which we continue to denote
by the same symbol $CW(\varphi)$.
We also continue to denote the image of $\partial_m$ in $\CC$
by the same symbol.

We give a structure of a colored CW-complex to $CW(\varphi)$,
in the sense of Definition \ref{colored complex} (1) below,
by declaring that the vertices and the faces
which lie between $\partial_{m}$ and $\partial_{m+1}$
are white or gray according as $m$ is even or odd.
Thus $CW(\varphi)$ admits a structure of
colored $\langle D \rangle$-CW-complex
in the sense of Definition \ref{colored complex} (3) below.

\begin{definition}
\label{colored complex}
\rm
(1) By a \textit{colored CW-complex}, we mean a
$2$-dimen\-sional
CW-complex in which each vertex and
each open $2$-cell is assigned a value in the set
$\{\mbox{white, gray}\}$.

(2) A \textit{color-preserving CW-isomorphism}
is a homeomorphism between colored CW-complexes
which carries cells homeomorphically to cells,
and respects the colorings.
Two colored CW-complexes are said to be \textit{isomorphic}
if there is a color-preserving CW-isomorphism between them.

(3) For a group $G$, a \textit{colored $G$-CW-complex}
is a colored CW-complex $W$ together with an action by $G$
such that each element of $G$ acts as a color-preserving
$CW$-automorphism of $W$.
An \textit{isomorphism} between two colored $G$-CW-complexes
is defined to be a $G$-equivariant
color-preserving CW-isomorphism between them.
\end{definition}

The following
result
shows that this colored CW-complex
$CW(\varphi)$ reflects the way the Cannon-Thurston map $\CT$
fills in $\CCC$ (see \cite[Introduction and Proposition 7.7]{Cannon-Dicks2}).

\begin{proposition}
The image of the open interval $(w_m,w_{m+1})$
by the Cannon-Thurston map $\CT$ is
equal to the closed region bounded by
$\partial_{m}$ and $\partial_{m+1}$.
Moreover, $\CT$ fills in the region
(gradually) from top to bottom
or from bottom to top,
according as the region is white or gray.
\end{proposition}

Next, we give an explicit description of the vertex set
of $CW(\varphi)$.
Recall the
bi-infinite word $\Omega=\prod f_n$
in the letters $\{R,L\}$ in Section \ref{sec.farey}.

\begin{definition}
\label{def.recursiveFF}
\rm
Let $\FF_n$ $(n\in\ZZ)$ be the bi-infinite sequence
of automorphisms
of $\pi_1(\OO)$
defined by the following rules:
\begin{equation*}
\FF_0 = 1, \quad
\FF_n= \FF_{n-1}f_n.
\end{equation*}
\end{definition}
Then the following proposition holds
(see \cite[Definition 7.13]{Cannon-Dicks2}).

\begin{proposition}
\label{prop.CW-vertices}
{\normalfont (1)} The vertex set of $\partial_0$
is equal to
\[
\{\rho_{\CCC}(\FF_n(B))(\infty)\}_{n\in\ZZ},
\]
in increasing order.
More generally,
the vertex set of $\partial_{2m}$
is equal to
\[
\{\rho_{\CCC}(\inn{D^m}\FF_n(B))(\infty)\}_{n\in\ZZ}
=\{m+\rho_{\CCC}(\FF_n(B))(\infty)\}_{n\in\ZZ}
\]
in increasing order.

{\normalfont (2)}
The vertex set of $\partial_1$
is equal to
\[
\{\rho_{\CCC}(\FF_n \inn{C}(B))(\infty)\}_{n\in\ZZ},
\]
in increasing order.
More generally,
the vertex set of $\partial_{2m+1}$
is equal to
\[
\{\rho_{\CCC}(\inn{D^m}\FF_n\inn{C}(B))(\infty)\}_{n\in\ZZ}
=\{m+\rho_{\CCC}(\FF_n\inn{C}(B))(\infty)\}_{n\in\ZZ}
\]
in increasing order.
\end{proposition}

\begin{remark}\rm
(1) In the above proposition,
$B$ and $C$ are elliptic generators
as  in Lemma \ref{lemma.RL},
and the identities among the sets follow from the
identity (\ref{trasition-equation}).

(2)
In \cite[Definition 7.13]{Cannon-Dicks2},
the vertex
\[
\rho_{\CCC}(\inn{D^m}\FF_n\inn{C}(B))(\infty)
=\rho_{\CCC}(\FF_n\inn{C}(B))(\infty)+m
\]
is described as
\[
\rho_{\CCC}(\inn{D^{m+1}}\FF_n(AC))(\infty)
=\rho_{\CCC}(\inn{D}\FF_n(AC))(\infty)+m.
\]
These points coincide,
because
\begin{align*}
\rho_{\CCC}(\inn{D}\FF_n(AC))(\infty)
&=\rho_{\CCC}(\FF_n\inn{D}(AC))(\infty)\\
&=\rho_{\CCC}(\FF_n(CBCD^{-1}))(\infty)\\
&=\rho_{\CCC}(\FF_n\inn{C}(B))(\infty).
\end{align*}
\end{remark}

In order to describe
how the adjacent vertical lines
$\partial_m$ and $\partial_{m+1}$ intersect each other,
we prepare the following notation.
\begin{definition}
\label{definition.P(m,n)}
\rm

(1) For each $(m,n)\in\ZZ^2$,
let $P_{m,n}$ be the elliptic generator defined by the following formulas:
\[
P_{2m,n}:= \inn{D^m}\FF_{n-1}(B) ,\quad
P_{2m+1,n}:= \inn{D^m}\FF_{n} \inn{C}(B).
\]

(2) For each $(m,n)\in\ZZ^2$,
let $p_{m,n}$ be the point in $\CC$ defined by the following formula:
\[
p_{m,n}:=\rho_{\CCC}(P_{m,n})(\infty).
\]

(3) For each $n\in\ZZ$, set:
\begin{eqnarray}
n_+&:=&\min\{k\in\ZZ\mid \mbox{$f_k=f_n$ and $k>n$}\}, \nonumber
\\
n_-&:=&\max\{k\in\ZZ\mid \mbox{$f_k=f_n$ and $k<n$}\}. \nonumber
\end{eqnarray}
It should be noted that $(n_+)_-=n=(n_-)_+$ and
$f_{n_+}=f_n=f_{n_-}$.
\end{definition}

Then we have the following
(see \cite[Definition 7.13 (iii)]{Cannon-Dicks2}).

\begin{proposition}
\label{prop.verticesCW}
{\normalfont(1)} For each $m\in\ZZ$, the set $\{p_{m,n}\}_{n\in\ZZ}$
forms the vertex set of $\partial_{m}$.

{\normalfont(2)} The intersection of two mutually adjacent
elements of $\{\partial_{m}\}$
is as follows.
\begin{align*}
 \partial_{2m-1}\cap \partial_{2m}
&&=&& \{p_{2m,n}\mid f_{n}=L\} &&=&& \{p_{2m-1,n}\mid f_n=L\},&&
\\
\partial_{2m}\cap \partial_{2m+1}
&&=&& \{p_{2m,n}\mid f_{n}=R\} &&=&& \{p_{2m+1,n}\mid f_n=R\}.&&
\end{align*}
Moreover, we have the following identities
among the vertices:
\[
p_{2m,n}
=
\begin{cases}
p_{2m+1,n_+}
& \mbox{if $f_n=R$,}\\
p_{2m-1,n_+}
& \mbox{if $f_n=L$.}
\end{cases}
\]
Equivalently, we have the following:
\[
p_{2m+1,n}
=
\begin{cases}
p_{2m,n_-}
& \mbox{if $f_n=R$,}\\
p_{2m+2,n_-}
& \mbox{if $f_n=L$.}
\end{cases}
\]
\end{proposition}

In order to give an explicit combinatorial model
of $CW(\varphi)$, we introduce the following definition
(see \cite[Definition 4.2]{Cannon-Dicks2} and Figure \ref{fig.CW'}).

\begin{definition}
\label{def.CW'}
\rm
Let $CW'(\varphi)$ be the CW-decomposition of $\RR^2$
defined as follows:
The vertex set is
$\ZZ^2$, where each vertex $(m,n)$ is endowed with
the label $f_n\in\{L,R\}$.
The edge set consists of
the \textit{vertical} edges and
the \textit{slanted} edges,
which are defined as follows:
\begin{enumerate}[(E1)]
\item
The vertical edges are
\[
\langle (m,n), (m,n+1) \rangle.
\]
\item
The slanted edges are
\begin{align*}
\langle (2m,n), (2m+1,n_+) \rangle
& \quad
\mbox{if $f_n=R$,}\\
\langle (2m,n), (2m-1,n_+) \rangle
& \quad
\mbox{if $f_n=L$.}
\end{align*}
\end{enumerate}

The face set of $CW'(\varphi)$ is
$\{c'_{m,n}\}_{(m,n)\in\ZZ^2}$,
where $c'_{m,n}$
is described as follows.
\begin{enumerate}[(F1)]
\item
If $f_n=R$ (and hence $f_{n_{\pm}}=R$),
then $c'_{m,n}$ is the convex hull of
\[
\{(2m,n_-), (2m,n), (2m+1,n), (2m+1,n_+)\},
\]
and we assign  the color \lq white' to its interior
(see Figure \ref{fig.CW'}).
\item
If $f_n=L$ (and hence $f_{n_{\pm}}=L$),
then $c'_{m,n}$ is the convex hull of
\[
\{(2m-1,n), (2m-1,n_+), (2m,n_-), (2m,n)\},
\]
and we assign the color \lq gray' to its interior.
\end{enumerate}
Thus the interior of each $2$-cell  of $CW'(\varphi)$
has the color white or gray
according as it lies in
the vertical strip  $[m, m+1]\times\RR$
with $m$ even or odd.
It should be noted that the colored complex
$CW'(\varphi)$ admits the action of the infinite
cyclic group
$\langle D\rangle$
by setting $D(x,y)=(x+2,y)$.
\end{definition}

\begin{figure}[p]
\begin{center}
\setlength{\unitlength}{1truecm}
\begin{picture}(5.36,17.3)
\put(0,.3){\epsfig{file=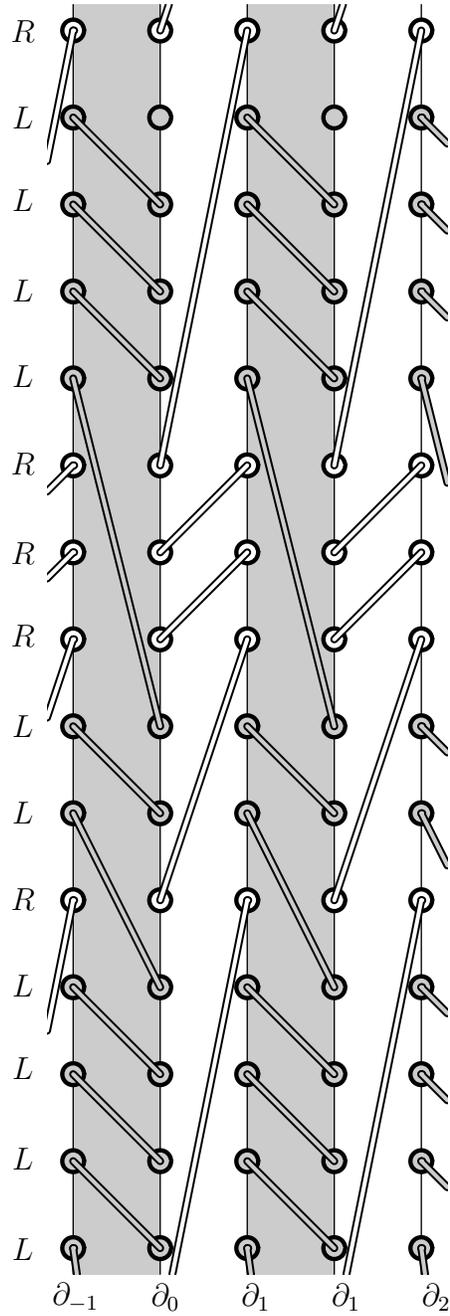, width =5.36truecm}}
\put(0.4,0){\makebox(0,0)[c]{$\partial_{-1}$}}
\put(1.6,0){\makebox(0,0)[c]{$\partial_{0}$}}
\put(2.8,0){\makebox(0,0)[c]{$\partial_{1}$}}
\put(4,0){\makebox(0,0)[c]{$\partial_{1}$}}
\put(5.2,0){\makebox(0,0)[c]{$\partial_{2}$}}
\put(-.3,0.65){\makebox(0,0)[c]{$L$}}
\put(-.3,1.85){\makebox(0,0)[c]{$L$}}
\put(-.3,3.05){\makebox(0,0)[c]{$L$}}
\put(-.3,4.15){\makebox(0,0)[c]{$L$}}
\put(-.3,5.3){\makebox(0,0)[c]{$R$}}
\put(-.3,6.45){\makebox(0,0)[c]{$L$}}
\put(-.3,7.6){\makebox(0,0)[c]{$L$}}
\put(-.3,8.8){\makebox(0,0)[c]{$R$}}
\put(-.3,9.9){\makebox(0,0)[c]{$R$}}
\put(-.3,11.1){\makebox(0,0)[c]{$R$}}
\put(-.3,12.25){\makebox(0,0)[c]{$L$}}
\put(-.3,13.4){\makebox(0,0)[c]{$L$}}
\put(-.3,14.6){\makebox(0,0)[c]{$L$}}
\put(-.3,15.7){\makebox(0,0)[c]{$L$}}
\put(-.3,16.85){\makebox(0,0)[c]{$R$}}
\end{picture}
\caption{The CW-complex $CW'(\varphi)$.
The vertical line segments represent vertical edges and
the double line segments represent slanted edges}\label{fig.CW'}
\end{center}
\end{figure}

In the above definition,
the vertex $(m,n)$ of
$CW'(\varphi)$ corresponds to
the vertex $p_{m,n}$ of $CW(\varphi)$, and the
slanted edges of $CW'(\varphi)$ correspond to
the identities in Proposition \ref{prop.verticesCW}(2).
This motivates us to define yet another CW-complex,
$CW^*(\varphi)$, as
the colored CW-complex obtained from $CW'(\varphi)$ by collapsing
each closed slanted edge in $CW'(\varphi)$ to a point.
The image of the $2$-cell $c'_{m,n}$
is a $2$-cell of $CW^*(\varphi)$,
denoted by $c^*_{m,n}$, and
the set $\{c^*_{m,n}\}_{(m,n)\in\ZZ}$
is the set of the $2$-cells of $CW^*(\varphi)$.
The open $2$-cells of $CW^*(\varphi)$ inherit colors
from those of $CW'(\varphi)$.
Each vertex of $CW^*(\varphi)$ is the image of
a slanted edge of $CW'(\varphi)$,
and it is endowed with the color white or gray
according as the slanted edge is contained in the white or gray
strip.
This determines the colored structure on $CW^*(\varphi)$.
The action of $\langle D\rangle$ on $CW'(\varphi)$
descends to an action of $\langle D\rangle$ on $CW^*(\varphi)$.
Thus $CW^*(\varphi)$ is a colored $\langle D\rangle$-CW-complex,
and it gives a combinatorial model of
the
colored $\langle D\rangle$-CW-complex $CW(\varphi)$.
To be precise, we have the following theorem
(see \cite[Theorem 7.12]{Cannon-Dicks2}).

\begin{theorem}
\label{Th.CDTessellation}
There is a color-preserving
$\langle D\rangle$-isomorphism
from $CW^*(\varphi)$ to $CW(\varphi)$,
sending the vertex $[(m,n)]$ to the vertex $p_{m,n}$.
The isomorphism maps the image of the vertical line
$\{m\}\times\ZZ$
in $CW'(\varphi)$
to the vertical line
$\partial_{m}$ in $CW(\varphi)$.
Here $[(m,n)]$ denotes the vertex of $CW^*(\varphi)$
represented by the vertex $(m,n)$ of $CW'(\varphi)$.
\end{theorem}

\begin{remark}
\label{remark.CW-vertices}
\rm
Since $\rho_{\CCC}$ is a faithful discrete representation,
two elliptic generators $P_{m,n}$ and $P_{m',n'}$ coincide
if and only if the two points $p_{m,n}$ and $p_{m',n'}$ coincide.
Thus the slanted edges of $CW'(\varphi)$ represent identities
among the elliptic generators, too (see Lemma \ref{lemma.P(m,n)vsQ(m,n)2}(3)).
Thus we may identify the vertex set of $CW^*(\varphi)$
with the set $\{P_{m,n}\}$ of elliptic generators,
by the correspondence $[(m,n)]\mapsto P_{m,n}$.
\end{remark}

Let $c_{m,n}$ be the $2$-cell of $CW(\varphi)$ obtained as
the homeomorphic image of $c^*_{m,n}$.
Then $\{c_{m,n}\}_{(m,n)\in\ZZ}$
is the set of the $2$-cells of $CW(\varphi)$.
In the remainder of this section,
we give a description of $c_{m,n}$.
To this end, let
\begin{align*}
\FF_n'
&=
\begin{cases}
\FF_{n-1}L
& \mbox{if $f_n=R$,}\\
\FF_{n-1}R
& \mbox{if $f_n=L$.}
\end{cases}
\end{align*}
Then $\FF_n'$ is an automorphism of $\pi_1(\OO)$
preserving the distinguished element $D$.
Thus it induces a $\langle D\rangle$-equivariant homeomorphism of $\End$
preserving the equivalence relation $\sim$ defined by
(\ref{fuchsian-equivalence-relation}).
Since the limit set
$\Lambda(\rho_{\RRR})=\RRR$ is identified with
$\End/{\sim}$,
$\FF_n'$ induces a $\langle D\rangle$-equivariant self-homeomorphism
of $\Lambda(\rho_{\RRR})=\RRR$,
which we continue to denote by the same symbol.
Then we have the following description
of the $2$-cells of $CW(\varphi)$
(\cite[Theorem 7.12]{Cannon-Dicks2}).

\begin{proposition}
The $2$-cell $c_{m,n}\mkern-1.5mu$ of $ CW(\varphi)\mkern-1.5mu$
is equal to
$\CT(\inn{D^m}\FF_n'([B]))$.
\end{proposition}

In the above proposition, $B$ is
the elliptic generator
of
Lemma \ref{lemma.RL},
and $[B]$ denotes the image, in $\RRR=\End/{\sim}$,
of the subset $[B]$ of the space $\End$,
consisting of those infinite words
which
have
$B$ as an initial segment.

\section{Statement and the proof of the main result}
\label{sec.statement}

\begin{theorem}
\label{maintheorem}
The vertex set of the fractal tessellation $CW(\varphi)$
is identical with the vertex set of the
projected horosphere
triangulation $\Delta(\varphi)$.
Moreover,
the combinatorial structure of the colored
$\langle D \rangle$-CW-complex
$CW(\varphi)$
can be recovered from that of
the layered $\langle D \rangle$-simplicial complex
$$(\Delta(\varphi), \{\LL(\rho_{\CCC},\sigma_n)\}),$$
and vice versa.
\end{theorem}

An explicit description of the
second statement of the theorem is
presented in italics in
mid-course
of the proof.

In order to give a proof of the main theorem,
we
recall Convention \ref{convention.Farey}.
Then we have the following lemma
for the slope $s(P_{m,n})$ of
the elliptic
generator $P_{m,n}$ of
Definition
\ref{definition.P(m,n)}.

\begin{lemma}
\label{lemma.action of FF}
{\normalfont(1)} $\sigma_n=(\FF_{n})_*(\sigma_0)$ and
$\sigma_{n+1}=(\FF_{n})_*(\sigma_1)$.

{\normalfont(2)} $s(P_{2m,n}) = (\FF_{n-1})_*(1)$, and it is equal to
the vertex of $\sigma_n$
which is not contained in $\sigma_{n-1}$.

{\normalfont(3)} $s(P_{2m+1,n}) = (\FF_n)_*(-1)$, and
it is equal to the vertex of $\sigma_n$
which is not contained in $\sigma_{n+1}$.
\end{lemma}

\begin{proof}
(1) If $n=0$, then $\FF_0=id$ and the assertion obviously holds.
Assume that
$n > 0$ and
the assertion holds for $n-1$.
Since
$\FF_n=(\FF_{n-1}f_n\FF_{n-1}^{-1})\FF_{n-1}$,
the inductive hypothesis
implies
\[
((\FF_n)_*(\sigma_0),(\FF_n)_*(\sigma_1))
=((\FF_{n-1}f_n\FF_{n-1}^{-1})_*(\sigma_{n-1}),
(\FF_{n-1}f_n\FF_{n-1}^{-1})_*(\sigma_{n})).
\]
Since
$(\FF_{n-1})_*$ maps the pair $(\sigma_0,\sigma_1)$
to the pair $(\sigma_{n-1},\sigma_{n})$,
the conjugate $(\FF_{n-1}f_n\FF_{n-1}^{-1})_*$
of $(f_n)_*$ by $(\FF_{n-1})_*$ rotates the pair
$(\sigma_{n-1},\sigma_{n})$ to the right or left,
according as $f_n$ is $R$ or $L$.
By the definition of $f_n$,
this implies that $(\FF_{n-1}f_n\FF_{n-1}^{-1})_*$
maps the pair $(\sigma_{n-1},\sigma_{n})$
to $(\sigma_{n},\sigma_{n+1})$.
Thus the assertion is valid for $n$.
This proves the assertion when $n$ is non-negative.
A parallel argument works for the case when $n$ is negative.

(2) and (3). These assertions follow from (1)
and the fact that
$1$ (resp. $-1$) is the vertex of $\sigma_1$ (resp. $\sigma_0$)
which is not contained in $\sigma_0$ (resp. $\sigma_1$).
\end{proof}

\begin{definition}
\label{def.Q(m,n)}
\rm
Let $\{Q_{m,1}\}_{m\in\ZZ}$ denote the sequence of elliptic generators
associated with $\sigma_1$ such that
\[
(Q_{0,1},Q_{1,1},Q_{2,1})=(A,B,C).
\]
Let $Q_{m,n}$ be the elliptic generator defined by
\[
Q_{m,n}=\FF_{n-1}(Q_{m,1}).
\]
(The second identity is consistent with the definition
of $Q_{m,1}$ even when $n=1$, because $\FF_0=id$.)
\end{definition}

\begin{lemma}
\label{lemma.Q(m,n)}
{\normalfont(1)} For each $n\in\ZZ$,
$\{Q_{m,n}\}_{m\in\ZZ}$ is the sequence of elliptic generators
associated with $\sigma_n$.
Thus $\{Q_{m,n}\}_{m\in\ZZ}$ is equal to $\{P_m^{(n)}\}_{m\in\ZZ}$
as
in Theorem \ref{thm.canonical-decomposition},
after a shift of indices.

{\normalfont(2)} For each $n\in\ZZ$, the sequences $\{Q_{m,n}\}_{m\in\ZZ}$ and
$\{Q_{m,n+1}\}_{m\in\ZZ}$ are related as follows.
\begin{enumerate}[\normalfont (i)]
\item
If $f_n=R$,
then
\[
Q_{3m,n}=Q_{3m,n+1}, \quad Q_{3m+1,n}=Q_{3m+2,n+1}.
\]
\item
If $f_n=L$,
then
\[
Q_{3m+1,n}=Q_{3m,n+1}, \quad Q_{3m+2,n}=Q_{3m+2,n+1}.
\]
\end{enumerate}
Moreover, these are the only identities among the members of
the two sequences.

{\normalfont(3)} Two elements $Q_{m,n}$ and $Q_{m',n'}$ are identical
if and only if they are related by a finite sequence
of the above identities.
To be precise, $Q_{m,n}=Q_{m',n'}$ if and only if
$(m,n)$ and $(m',n')$ are equivalent
with respect to the equivalence relation
generated by the following elementary relations:
\begin{enumerate}[\normalfont (i)]
\item
If $f_n=R$, then
\[
(3m,n)\sim (3m,n+1), \quad (3m+1,n) \sim (3m+2,n+1).
\]
\item
If $f_n=L$,
then
\[
(3m+1,n)\sim (3m,n+1), \quad (3m+2,n)\sim (3m+2,n+1).
\]
\end{enumerate}
\end{lemma}

\begin{proof}
(1) Since $\{Q_{m,1}\}_{m\in\ZZ}$ is the sequence of elliptic generators
associated with $\sigma_1$,
$\{\FF_{n-1}(Q_{m,1})\}_{m\in\ZZ}$
is the sequence of elliptic generators
associated with $(\FF_{n-1})_*(\sigma_1)$,
which is equal to $\sigma_n$
by Lemma \ref{lemma.action of FF}(1).
Hence $\{Q_{m,n}\}_{m\in\ZZ}$ is the sequence of elliptic generators
associated with $\sigma_n$.

(2)
Since $\FF_n=(\FF_{n-1}f_n\FF_{n-1}^{-1})\FF_{n-1}$,
we have
\begin{align*}
Q_{k,n+1}
&=
\FF_n(Q_{k,1})\\
&=
(\FF_{n-1}f_n\FF_{n-1}^{-1})\FF_{n-1}(Q_{k,1})\\
&=
\FF_{n-1}f_n\FF_{n-1}^{-1}(Q_{k,n}).
\end{align*}
Since $\FF_{n-1}$ maps $Q_{k,1}$ to $Q_{k,n}$,
this implies that
the triple
$$(Q_{3m,n+1},Q_{3m+1,n+1},Q_{3m+2,n+1})$$
is obtained from the triple
$(Q_{3m,n},Q_{3m+1,n},Q_{3m+2,n})$
by applying the first or second rule in (\ref{DefiningRL})
according as $f_n=R$ or $L$,
where $(A,B,C)$ is replaced with the triple
$(Q_{3m,n},Q_{3m+1,n},Q_{3m+2,n})$.
Namely,
\begin{align*}
& (Q_{3m,n+1},Q_{3m+1,n+1},Q_{3m+2,n+1})\\
& =
\begin{cases}
(Q_{3m,n},\inn{Q_{3m+1,n}} (Q_{3m+2,n}),Q_{3m+1,n}) &
\mbox{if $f_n=R$},\\
(Q_{3m+1,n},\inn{Q_{3m+1,n}}(Q_{3m,n}),Q_{3m+2,n}) &
\mbox{if $f_n=L$}.
\end{cases}
\end{align*}
Hence we obtain (2).

(3)
Since the `if' part is obvious,
we prove the `only if' part.
Suppose $Q_{m,n}=Q_{m',n'}$,
and denote this elliptic generator by $Q$.
We may assume $m'=m+r$ for some non-negative integer $r$.
Then, for every $i$ ($0\le i\le r$),
the slope $s(Q)$ is a vertex of the Farey triangle $\sigma_{n+i}$
and hence $Q$ belongs to the sequence of elliptic generators
associated with $\sigma_{n+i}$.
Thus we can find by (2) a sequence $\{m_i\}_{0\le i\le r}$
of integers such that $m_0=m$ and
\[
(m,n)=(m_0,n)\sim (m_1,n+1) \sim \cdots \sim (m_r, n+r).
\]
Since $Q_{m',n'}=Q_{m,n}=Q_{m_r, n+r}=Q_{m_r, n'}$,
we have $m'=m_r$,
because the sequence of elliptic generators associated with
any Farey triangle consists of mutually distinct elements.
Hence $(m,n)\sim (m_r, n+r)=(m',n')$.
Thus we obtain the desired result.
\end{proof}

The following lemma describes the relation between
the elliptic generators $\{P_{m,n}\}$ in Definition \ref{definition.P(m,n)}
and $\{Q_{m,n}\}$ in Definition \ref{def.Q(m,n)}.

\begin{lemma}
\label{lemma.P(m,n)vsQ(m,n)}
The sets
$\{P_{m,n}\}_{(m,n)\in \ZZ^2}$ and $\{Q_{m,n}\}_{(m,n)\in \ZZ^2}$
are identical. To be precise, the following hold.

{\normalfont(1)}
$P_{2m,n}=Q_{3m+1,n}$ and $P_{2m+1,n}=Q_{3m+d(n),n}$,
where
\[
d(n)=
\begin{cases}
2 & \mbox{if $f_n=R$}, \\
3 & \mbox{if $f_n=L$}.
\end{cases}
\]

{\normalfont(2)} Conversely, the following hold.
\begin{align*}
Q_{3m+1,n} &= P_{2m,n},\\
Q_{3m+2,n} &=
\begin{cases}
P_{2m+1,n} & \mbox{if $f_n=R$},\\
P_{2m+1,n+r} & \mbox{if $f_n=L$},
\end{cases}
\\
Q_{3m+3,n} &=
\begin{cases}
P_{2m+1,n+r} & \mbox{if $f_n=R$},\\
P_{2m+1,n+1} & \mbox{if $f_n=L$},
\end{cases}
\end{align*}
where $r$ is the smallest positive integer such that $f_{n+r}\ne f_n$.
\end{lemma}

\begin{proof}
(1)
Since
\[
P_{2m+k,n}= \inn{D^m}(P_{k,n}), \quad
Q_{3m+k,n} = \inn{D^m}(Q_{k,n}),
\]
we have only to prove the identities when $m=0$.

The first identity is proved as follows.
\[
P_{0,n}=\FF_{n-1}(B)=\FF_{n-1}(Q_{1,1})=Q_{1,n}.
\]

We prove the second identity
for $n \ge 0$
by induction on $n$.
The identity for $n=0$ is proved as follows.
\begin{gather*}
P_{1,0}=\FF_0\inn{C}(B)=\inn{C}(B) =L^{-1}\inn{D}(A)
=L^{-1}(Q_{3,1})\\
=\FF_{-1}(Q_{3,1})=Q_{3,0}=Q_{d(0),0}.
\end{gather*}
In the above, the third identity follows from Lemma \ref{lemma.RL}(3),
the fifth identity follows
from the fact that $\FF_{-1}=f_0^{-1}=L^{-1}$ (see Convention
\ref{convention.Farey}),
and the last identity follows from the convention $f_0=L$.
Now suppose the desired identity holds for some
non-negative integer $n$, namely $P_{1,n}=Q_{d(n),n}$.
Then
\begin{align}
\label{InductiveIdentity}
P_{1,n+1}
&=
\FF_{n+1}\inn{C}(B)
=
 (\FF_nf_{n+1}\FF_n^{-1}) \FF_n \inn{C}(B)
\\
& =
 \FF_nf_{n+1}\FF_n^{-1} (P_{1,n})
=
 \FF_nf_{n+1}\FF_n^{-1} (Q_{d(n),n}) \notag \\
&=
\FF_nf_{n+1}f_n^{-1}\FF_{n-1}^{-1}(Q_{d(n),n})
=
\FF_nf_{n+1}f_n^{-1}(Q_{d(n),1}). \notag
\end{align}
If $f_{n+1}=f_n$, then $d(n)=d(n+1)$ and hence the last element in
the above identity is equal to
\[
\FF_n(Q_{d(n),1})=Q_{d(n),n+1}=Q_{d(n+1),n+1},
\]
and therefore the desired identity holds for $n+1$.
Suppose $(f_n,f_{n+1})=(R,L)$.
By Lemma \ref{lemma.RL}(3), we have
\[
f_{n+1}f_n^{-1}(Q_{m,1})=LR^{-1}(Q_{m,1})=Q_{m+1,1}.
\]
Hence the last term of (\ref{InductiveIdentity}) is equal to
\[
\FF_n(Q_{d(n)+1,1})=Q_{d(n)+1,n+1}=Q_{d(n+1),n+1}.
\]
Here the last identity follows from
the fact that $(d(n),d(n+1))=(2,3)$.
Thus the desired identity
holds
for $n+1$.
The case $(f_n,f_{n+1})=(L,R)$ is treated similarly.
Thus we have proved, by induction,
the desired identity for every non-negative integer $n$.
By a parallel argument, we can also show that the identity holds
for every integer $n$.

(2)
Suppose $f_n=R$.
Then, by Lemma \ref{lemma.P(m,n)vsQ(m,n)}(1), we have
\[
Q_{3m+1,n}= P_{2m,n},\quad
Q_{3m+2,n}= P_{2m+1,n}.
\]
By Lemma \ref{lemma.Q(m,n)}(2),
we have $Q_{3m+3,n}=Q_{3m+3,n+1}$.
By applying the same lemma repeatedly, we have
\[
Q_{3m+3,n} = Q_{3m+3,n+1}= \cdots
=Q_{3m+3,n+r},
\]
because $f_n=\cdots = f_{n+r-1}=R$.
Since $f_{n+r}=L$, Lemma \ref{lemma.P(m,n)vsQ(m,n)}(1) implies
$Q_{3m+3,n+r}=P_{2m+1,n+r}$.
Hence we obtain the desired identity.
A similar argument also works for the case $f_n=L$.
\end{proof}

The following lemma completely describes
the identities among the elliptic generators
$\{P_{m,n}\}$ and $\{Q_{m,n}\}$.

\begin{lemma}
\label{lemma.P(m,n)vsQ(m,n)2}
{\normalfont(1)} If $f_n=R$,
then
$P_{2m,n}=P_{2m+1,n_+}$.
Moreover $Q_{m',n'}$ is equal to this element
if and only if $(m',n')$ belongs to the following set.
\[
\{(3m+1,n)\}\cup\{(3m+2,k) \mid n+1\le k\le n_+\}.
\]

{\normalfont(2)} If $f_n=L$, then $P_{2m,n}=P_{2m-1,n_+}$.
Moreover $Q_{m',n'}$ is equal to this element
if and only if $(m',n')$ belongs to the following set.
\[
\{(3m+1,n)\}\cup\{(3m,k)\mid n+1\le k\le n_+\}.
\]

{\normalfont(3)} The following are the only identities among
the $P_{m,n}$.
\[
P_{2m,n}
=
\begin{cases}
P_{2m+1,n_+}
& \mbox{if $f_n=R$,}\\
P_{2m-1,n_+}
& \mbox{if $f_n=L$.}
\end{cases}
\]
\end{lemma}

\begin{proof}
(1) Suppose $f_n=R$. Then by Lemmas \ref{lemma.P(m,n)vsQ(m,n)}(1) and
\ref{lemma.Q(m,n)}(2), we see
\[
P_{2m,n}=Q_{3m+1,n}=Q_{3m+2,n+1}.
\]
Since $f_k=L$ for every $k$ ($n+1\le k\le n_+-1$),
Lemma \ref{lemma.Q(m,n)}(2) implies
that the above element is equal to $Q_{3m+2,k}$
for every $k$ ($n+1\le k\le n_+$).
Moreover, by Lemma \ref{lemma.Q(m,n)}(3),
we see that $Q_{m',n'}$ is equal to $Q_{3m+1,n}$
if and only if $(m',n')$ appears in the above.
(Check that there are no elementary relations
relating $(3m+1,n)$ with $(*,n-1)$
and those relating $(3m,n_+)$ with $(*,n_++1)$.)
Since $f_{n_+}=f_n=R$,
we have $Q_{3m+2,n_+}=P_{2m+1,n_+}$
by Lemma \ref{lemma.P(m,n)vsQ(m,n)}(2).
This completes the proof of (1).

(2) is proved by an argument parallel to that of (1).

(3) The desired identity is already proved by (1) and (2).
Let $(r,s)$ and $(r',s')$ be elements of $\ZZ^2$
such that $P_{r,s}=P_{r',s'}$.
Suppose first that $r$ and $r'$ have the same parity.
Then we have $s=s'$ by Lemma \ref{lemma.action of FF}(2),(3).
So, by Definition \ref{definition.P(m,n)},
$P_{r',s'}=P_{r',s}$ is the conjugate of
$P_{r,s}$ by $D^{(r'-r)/2}$.
Since $P_{r,s}=P_{r',s'}$,
this implies $r=r'$ and hence $(r,s)=(r',s')$.
Suppose $r$ and $r'$ have different parity,
say $r=2m$ and $r'=2m'+1$ for some $m,m'\in\ZZ$.
Then, by Lemma \ref{lemma.action of FF}(2),(3) again,
$s'$ is uniquely determined by $s$.
On the other hand, we have $P_{r,s}=P_{2m,s}$ is equal to
$P_{2m+1,s_+}$ or $P_{2m-1,s_+}$ according as $f_s=R$ or $L$.
Thus we have $s'=s_+$.
By the argument for the same-parity case,
we see that $s'$ is equal to $2m+1$ or $2m-1$
according as $f_s=R$ or $L$.
Thus we obtain the conclusion.
\end{proof}

\begin{remark}
\rm
As observed in Remark \ref{remark.CW-vertices},
the assertion (3) in the above lemma is essentially equivalent to
the identities in Proposition \ref{prop.verticesCW}(2).
\end{remark}

Lemmas \ref{lemma.Q(m,n)} and \ref{lemma.P(m,n)vsQ(m,n)}
motivate us to introduce the following CW-decomposition
of $\RR^2$ (see Figure \ref{fig.CWD}).

\begin{definition}
\label{definition.parents-complex}
\rm
$\cellcomplex(\varphi)$ denotes the CW-decomposition
of $\RR^2$ defined as follows.
The vertex set is equal to $\ZZ^2$,
where $(m,n)$ is labeled with $Q_{m,n}$, and
the edge set consists of the \textit{horizontal edges},
the \textit{vertical edges} and the \textit{slanted edges},
which are described as follows.
\begin{enumerate}
\item
The \textit{horizontal edges} are:
\[
\langle (m,n), (m+1,n)\rangle.
\]
\item
The \textit{vertical edges} are:
\begin{align*}
&\langle (3m+1,n), (3m+1,n+1)\rangle, \\
&\langle (3m+d(n),n), (3m+d(n+1),n+1)\rangle.
\end{align*}

\item
For each $n\in\ZZ$,
the \textit{slanted edges} lying between
$\RR\times\{n\}$ and $\RR\times\{n+1\}$
are:

if $f_n=R$,
\begin{align*}
&\langle (3m,n),(3m,n+1)\rangle,\\
&\langle (3m+1,n),(3m+2,n+1)\rangle,
\end{align*}
and, if $f_n=L$,
\begin{align*}
&\langle (3m+1,n),(3m,n+1)\rangle,\\
&\langle (3m+2,n),(3m+2,n+1)\rangle.
\end{align*}
\end{enumerate}
\end{definition}

\begin{figure}[p]
\begin{center}
\setlength{\unitlength}{1truecm}
\begin{picture}(7,17)
\put(1,0.3){\epsfig{file=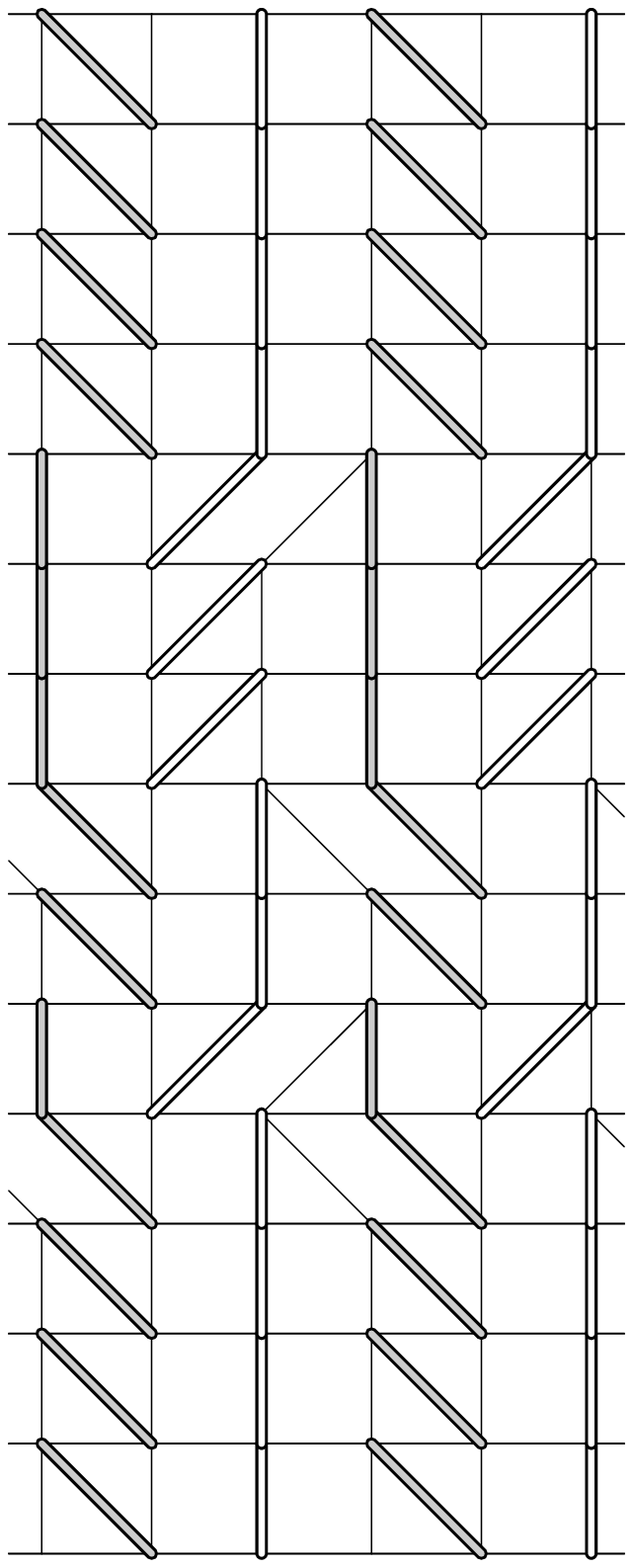, width =6.5truecm}}
\put(1.3,0){\makebox(0,0)[c]{$0$}}
\put(2.5,0){\makebox(0,0)[c]{$1$}}
\put(3.7,0){\makebox(0,0)[c]{$2$}}
\put(4.75,0){\makebox(0,0)[l]{$3$}}
\put(6,0){\makebox(0,0)[c]{$4$}}
\put(7.1,0){\makebox(0,0)[c]{$5$}}
\put(0.7,.7){\makebox(0,0)[r]{$-3$}}
\put(0.7,1.3){\makebox(0,0)[r]{$L$}}
\put(0.7,1.85){\makebox(0,0)[r]{$-2$}}
\put(0.7,2.45){\makebox(0,0)[r]{$L$}}
\put(0.7,3.0){\makebox(0,0)[r]{$-1$}}
\put(0.7,3.6){\makebox(0,0)[r]{$L$}}
\put(0.7,4.15){\makebox(0,0)[r]{$0$}}
\put(0.7,4.75){\makebox(0,0)[r]{$L$}}
\put(0.7,5.30){\makebox(0,0)[r]{$1$}}
\put(0.7,5.9){\makebox(0,0)[r]{$R$}}
\put(0.7,6.45){\makebox(0,0)[r]{$2$}}
\put(0.7,7.05){\makebox(0,0)[r]{$L$}}
\put(0.7,7.6){\makebox(0,0)[r]{$3$}}
\put(0.7,8.2){\makebox(0,0)[r]{$L$}}
\put(0.7,8.75){\makebox(0,0)[r]{$4$}}
\put(0.7,9.35){\makebox(0,0)[r]{$R$}}
\put(0.7,9.9){\makebox(0,0)[r]{$5$}}
\put(0.7,10.5){\makebox(0,0)[r]{$R$}}
\put(0.7,11.1){\makebox(0,0)[r]{$6$}}
\put(0.7,11.65){\makebox(0,0)[r]{$R$}}
\put(0.7,12.2){\makebox(0,0)[r]{$7$}}
\put(0.7,12.8){\makebox(0,0)[r]{$L$}}
\put(0.7,13.385){\makebox(0,0)[r]{$8$}}
\put(0.7,13.98){\makebox(0,0)[r]{$L$}}
\put(0.7,14.55){\makebox(0,0)[r]{$9$}}
\put(0.7,15.15){\makebox(0,0)[r]{$L$}}
\put(0.7,15.7){\makebox(0,0)[r]{$10$}}
\put(0.7,16.3){\makebox(0,0)[r]{$L$}}
\put(0.7,16.8){\makebox(0,0)[r]{$11$}}
\end{picture}
\caption{The CW-complex $\cellcomplex(\varphi)$.
The horizontal line segments represent horizontal edges,
the thin non-horizontal line segments represent vertical edges, and
the thick \, non-horizontal \, line \, segments \, represent \, slanted  edges.}\label{fig.CWD}
\end{center}
\end{figure}

For each $n\in\ZZ$,
the horizontal edges $\{\langle (m,n), (m+1,n)\rangle\}_{m\in\ZZ}$
of $\cellcomplex(\varphi)$
correspond to the edges of the layer
$\LL(\sigma_n)$ of $\Delta^*(\varphi)$
(cf. Proposition \ref{prop.model-Delta}(4)
and Lemma \ref{lemma.Q(m,n)}(1)).
For each $m\in\ZZ$,
the vertical edges
\[
\{\langle (3m+1,n), (3m+1,n+1)\rangle\}_{n\in\ZZ}
\]
correspond to the edges in the vertical line $\partial_{2m}$ in
$CW(\varphi)$,
and the vertical edges
\[
\{\langle (3m+d(n),n),
(3m+d(n+1),n+1)\rangle\}_{n\in\ZZ}
\]
correspond to the edges in the vertical line $\partial_{2m+1}$
in $CW(\varphi)$
(cf. Proposition \ref{prop.verticesCW}(2) and
Lemma \ref{lemma.P(m,n)vsQ(m,n)}(1)).
The slanted edges correspond to
the elementary relations in Lemma \ref{lemma.Q(m,n)}(3).
Thus $\cellcomplex(\varphi)$ can be regarded as
a
common parent
of the two tessellations $\Delta(\varphi)$ and $CW(\varphi)$.
To be precise, the following proposition holds.

\begin{proposition}
\label{prop.quotient-cellcomplex}
{\normalfont(1)} Let $\Delta^{**}(\varphi)$ be the CW-complex obtained from
$\cellcomplex(\varphi)$ by removing
$($the interior  of\,$)$ each  vertical edge  and collapsing
 each slanted edge to a point.
Then $\Delta^{**}(\varphi)$ is combinatorially isomorphic to
$\Delta^*(\varphi)$ $($and hence to $\Delta(\varphi))$,
where the vertex of $\Delta^{**}(\varphi)$
represented by $(m,n)$ corresponds to the vertex
$Q_{m,n}$ of $\Delta^*(\varphi)$.
Here, we employ the identification
described by Proposition \ref{prop.model-Delta}.

{\normalfont(2)} Let $CW^{**}(\varphi)$ be the CW-complex obtained from
$\cellcomplex(\varphi)$ by removing
$($the interior  of\,$)$ each horizontal edge  and collapsing
each slanted edge to a point.
Then $CW^{**}(\varphi)$ is combinatorially isomorphic to $CW^{*}(\varphi)$
$($and hence to $CW(\varphi))$,
where the vertex of $CW^{**}(\varphi)$
represented by $(3m+1,n)$ $($resp. $(3m+d(n),n))$
corresponds to the vertex $P_{2m,n}$ $($resp. $P_{2m+1,n})$
of $CW^{*}(\varphi)$.
Here, we employ the identification
described by Remark \ref{remark.CW-vertices}.
\end{proposition}

\begin{proof}
(1) By Proposition \ref{prop.model-Delta} and Lemma \ref{lemma.Q(m,n)}(1),
the vertex set of $\Delta^*(\varphi)$ is identified with the
set $\{Q_{m,n}\}_{(m,n)\in \ZZ^2}$.
On the other hand,
since the slanted edges of $\cellcomplex(\varphi)$
correspond to the elementary relations in
Lemma \ref{lemma.Q(m,n)}(3),
the vertex set of $\Delta^{**}(\varphi)$ is identified with
the set $\{Q_{m,n}\}_{(m,n)\in \ZZ^2}$.
Hence there is a natural bijection between the vertex sets of
$\Delta^*(\varphi)$ and $\Delta^{**}(\varphi)$.
By virtue of the characterization of the edge set and the face set
of $\Delta^*(\varphi)$
by Proposition \ref{prop.model-Delta},
the bijection between the vertex sets of
$\Delta^*(\varphi)$ and $\Delta^{**}(\varphi)$
extends
to an isomorphism
between the two CW-complexes.

(2)
By Remark \ref{remark.CW-vertices},
the vertex set of $CW^{*}(\varphi)$ is identified with the set
$\{P_{m,n}\}_{(m,n)\in \ZZ^2}$,
which in turn is equal to the set
$$\{Q_{m,n}\}_{(m,n)\in \ZZ^2}$$
by Lemma \ref{lemma.P(m,n)vsQ(m,n)}(1).
On the other hand, as in the proof of
Proposition \ref{prop.quotient-cellcomplex}(1),
the vertex set of $CW^{**}(\varphi)$ is also identified with
the set $\{Q_{m,n}\}_{(m,n)\in \ZZ^2}$.
Thus there is a natural bijection
between the vertex sets of $CW^*(\varphi)$ and $CW^{**}(\varphi)$.
By the correspondences mentioned in the paragraph
preceding
the proposition,
this bijection
extends
to an isomorphism
between the two CW-complexes.
\end{proof}

\begin{remark}
\rm
The subcomplex of $\cellcomplex(\varphi)$ obtained by
deleting the vertical edges is a tessellation by four types of pentagons.
These pentagons, which become triangles after slanted-edge collapsing,
can be thought of as being equal to the pentagons that appeared in
Figure~\ref{figure.EGS}.
\end{remark}

We now give a proof of the first assertion of Theorem \ref{maintheorem}.
By Theorem \ref{thm.canonical-decomposition}
and Lemma \ref{lemma.Q(m,n)}(1),
the vertex set of $\Delta(\varphi)$ is equal to
$\{\rho_{\CCC}(Q_{m,n})(\infty)\}_{(m,n)\in \ZZ^2}$.
By the definition of $CW(\varphi)$,
the vertex set of $CW(\varphi)$ is equal to
$\{\rho_{\CCC}(P_{m,n})(\infty)\}_{(m,n)\in \ZZ^2}$.
Since these two sets are identical
by Lemma \ref{lemma.P(m,n)vsQ(m,n)},
we obtain the first assertion of Theorem \ref{maintheorem}.

\vskip0.5cm

Next,
we show that
the combinatorial structure of the colored
$\langle D \rangle$-CW-complex
$CW(\varphi)$
can be
recovered from that of
the layered $\langle D \rangle$-simpli\-cial complex
$(\Delta(\varphi), \{\LL(\rho_{\CCC},\sigma_n)\})$.
In the following, we identify
$\Delta(\varphi)$
with $\Delta^{**}(\varphi)$
and $CW(\varphi)$ with $CW^{**}(\varphi)$.
We also employ the identification
described by Proposition \ref{prop.model-Delta}
and Remark \ref{remark.CW-vertices}.
Thus the vertices of $\Delta(\varphi)$ and
$CW(\varphi)$ are represented by elliptic generators,
and in particular,
the layer $\LL(\rho_{\CCC},\sigma_n)$ of $\Delta(\varphi)$
is identified with the following bi-infinite family
of edges of $\Delta^{**}(\varphi)$.
\[
\LL_n:=\{\langle Q_{m,n}, Q_{m+1,n}\rangle\}_{m\in\ZZ}
\]

\begin{definition}
\label{def.apex}
\rm
A vertex of $\LL_n$ is called an \textit{upward apex}
(resp. a \textit{downward apex})
if it is not a vertex of $\LL_{n-1}$ (resp. $\LL_{n+1}$)
(see Figure \ref{fig.trucated-tetrahedron}).
\end{definition}

\begin{lemma}
\label{lemma.apex}
{\normalfont(1)} For each $n\in\ZZ$,
a vertex of $\LL_n$ is an upward apex of $\LL_n$
if, and only if, it is equal to $P_{2m,n}$ for some $m\in\ZZ$.
Moreover,
$P_{2m,n+1}$ is a vertex of $\LL_{n+1}$
adjacent to $P_{2m,n}$ in $\LL_{n+1}$.
Furthermore, $P_{2m,n+1}$ is the unique upward apex of $\LL_{n+1}$
which is adjacent to $P_{2m,n}$ in $\Delta(\varphi)$.

{\normalfont(2)} For each $n\in\ZZ$,
a vertex of $\LL_n$ is
a
downward apex of $\LL_n$
if, and only if, it is equal to $P_{2m+1,n}$ for some $m\in\ZZ$.
Moreover, $P_{2m+1,n-1}$ is a vertex of $\LL_{n-1}$
adjacent to $P_{2m+1,n}$ in $\LL_{n-1}$.
Furthermore, $P_{2m+1,n-1}$ is the unique
downward
apex of $\LL_{n-1}$
which is adjacent to $P_{2m+1,n}$ in $\Delta(\varphi)$.
\end{lemma}

\begin{proof}
Since the proofs of (1) and (2) are parallel, we prove only (1). By
Lemma \ref{lemma.action of FF}, the slope of $P_{m,n}$ is the vertex
of $\sigma_n$ which is not contained in $\sigma_{n-1}$. Hence
$P_{m,n}$ is an upward apex of $\LL_n$. By the periodicity of
$\LL_n$ and $\LL_{n-1}$, every upward apex of $\LL_n$ is obtained
from a single upward apex, say $P_{0,n}$, by taking the conjugate of
$D^m$ for some $m\in\ZZ$. Thus it is equal to the following element
by Lemma \ref{lemma.P(m,n)vsQ(m,n)}(1).
\[
\inn{D^m}(P_{0,n}) = \inn{D^m}(Q_{1,n}) =Q_{3m+1,n}=P_{2m,n}
\]
Hence $\{P_{2m,n}\}_{m\in\ZZ}$ are the only upward apexes of $\LL_n$.

By Lemma \ref{lemma.P(m,n)vsQ(m,n)2}, we have
\[
P_{2m,n}=Q_{3m+1,n}=
\begin{cases}
Q_{3m+2,n+1} & \mbox{if $f_n=R$,}\\
Q_{3m,n+1} & \mbox{if $f_n=L$.}
\end{cases}
\]
Hence $P_{2m,n}$ is a vertex of $\LL_{n+1}$
adjacent to $Q_{3m+1,n+1}=P_{2m,n+1}$.
The other vertex of $\LL_{n+1}$
adjacent to $Q_{3m+1,n+1}$ in $\Delta(\varphi)$
is equal to $Q_{3m,n+1}$ or $Q_{3m+2, n+1}$
according as $f_n=R$ or $L$.
By Lemma \ref{lemma.Q(m,n)}(2),
it is equal to $Q_{3m,n}$ or $Q_{3m+2,n}$ accordingly.
Thus it is contained in $\LL_n$ and therefore
it is not an upward apex of $\LL_{n+1}$.
Hence we obtain the last assertion of (1).
\end{proof}

Lemma \ref{lemma.apex} implies that,
for each upward apex of $\LL_n$,
there is a unique upward apex of $\LL_{n+1}$
adjacent to it in $\Delta(\varphi)$.
Thus
\textit{
by successively joining mutually adjacent upward apexes,
we obtain a bi-infinite \lq vertical' edge path
in $\Delta(\varphi)$.
We call it an {\normalfont upward line}.
Similarly, we define a {\normalfont downward line} to be a
vertical edge path in $\Delta(\varphi)$
obtained by successively joining mutually adjacent
downward apexes.
}
By Lemmas \ref{lemma.apex} and \ref{lemma.P(m,n)vsQ(m,n)}(1),
the vertical line in $\cellcomplex(\varphi)$ with vertex set
$\{(3m+1,n)\}_{n\in\ZZ}$ projects homeomorphically onto
an upward line in $\Delta(\varphi)$, for each $m\in\ZZ$.
Similarly, the vertical line in $\cellcomplex(\varphi)$ with vertex set
$\{(3m+1,d(n))\}_{n\in\ZZ}$ projects homeomorphically onto a
downward line in $\Delta(\varphi)$, for each $m\in\ZZ$.
Since these
two kinds of
vertical lines in $\cellcomplex(\varphi)$
are mutually disjoint,
and since all upward lines and downward lines are obtained in this way,
an upward line and a downward line never cross each other
though they may share some common edges.
Moreover, upward lines and downward lines are located in $\CC$
alternately.
Hence
we may number them so that
\textit{
$\{\partial_{m}^{\Delta}\}_{m\in\ZZ}$ is the set of
upward/downward lines
located in $\Delta(\varphi)$
from left to right,
where $\partial_m^{\Delta}$ is an upward line or a downward line
according as $m$ is even or odd.
}
It should be noted that $D$ maps
$\partial_{m}^{\Delta}$ to $\partial_{m+2}^{\Delta}$
(see Figure \ref{fig.up-down-lines}).

\begin{figure}[t!]
\begin{center}
\setlength{\unitlength}{1truemm}
\begin{picture}(65,100)
\put(0,3){\epsfig{file=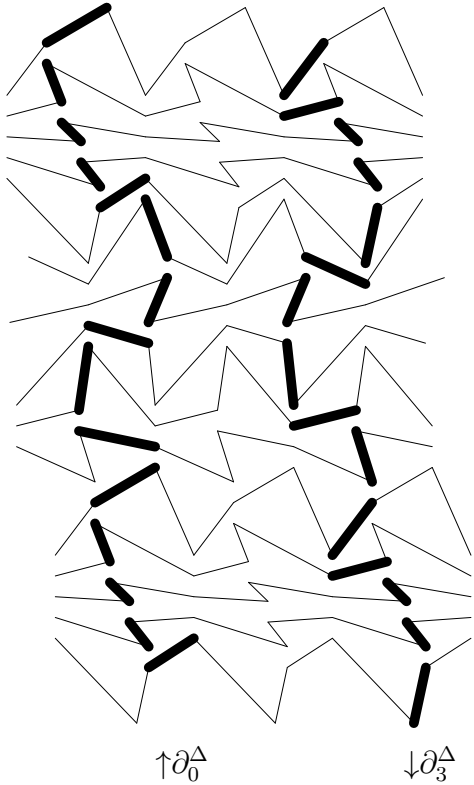, width =6.5truecm}}
\put(25,0){\makebox(0,0)[c]{${\uparrow}\partial_{0}^{\Delta}$}}
\put(58,0){\makebox(0,0)[c]{${\downarrow}\partial_{3}^{\Delta}$}}
\end{picture}
\caption{The upward/downward lines
$\{\partial_{m}^{\Delta}\}_{m\in\ZZ}$ in $\Delta(\varphi)$}\label{fig.up-down-lines}
\end{center}
\end{figure}

\textit{
The union of $\{\partial_{m}^{\Delta}\}_{m\in\ZZ}$
forms a $\langle D\rangle$-CW-decomposition of $\CC$.
We blow up each common edge of adjacent
upward/downward
lines into a bigon to produce a new $\langle D\rangle$-CW-decomposition,
$CW^{\#}(\varphi)$, and let $\partial_m^{\#}$ be the image of
$\partial_{m}^{\Delta}$ in $CW^{\#}(\varphi)$.
We color the interior of the $2$-cells in the new cell-complex
bounded by $\partial_{m}^{\#}$ and $\partial_{m+1}^{\#}$ white or gray
according as $m$ is even or odd.
We also color the vertices shared by $\partial_{m}^{\#}$ and
$\partial_{m+1}^{\#}$ white or gray by the same rule.
Then, by Proposition $\ref{prop.quotient-cellcomplex}(2)$,
we see that the resulting colored
$\langle D\rangle$-CW-complex
$CW^{\#}(\varphi)$ is isomorphic to the colored $\langle D\rangle$-CW-complex
$CW(\varphi)$.
}
Thus we have shown that
the combinatorial structure of the colored
$\langle D \rangle$-CW-complex
$CW(\varphi)$
can be
recovered from that of
the layered $\langle D \rangle$-simplicial complex
$(\Delta(\varphi), \{\LL_n\})$.

\vskip0.5cm

Next, we show that
the combinatorial structure of
the layered $\langle D \rangle$-simplicial complex
$(\Delta(\varphi), \{\LL_n\})$
can be
recovered from that of the colored
$\langle D \rangle$-CW-complex
$CW(\varphi)$.

\begin{lemma}
\label{lemma.2-cell-boundary}
{\normalfont(1)} Suppose $f_n=R$.
Then the open $2$-cell $c_{m,n}$ of $CW(\varphi)$
has the color white and its boundary is the union of
an edge path, $\alpha$, of $\partial_{2m}$ and
an edge path, $\beta$, of $\partial_{2m+1}$
which satisfy the following conditions.
\begin{enumerate}[\normalfont (i)]
\item
$\alpha\cap\beta=\partial\alpha=\partial\beta$
and it consists of the following two vertices.
\begin{align*}
& v_+:=P_{2m,n}=P_{2m+1,n_+}, \\
& v_-:=P_{2m,n_-}=P_{2m+1,n}.
\end{align*}
Moreover, the inverse image of $v_+$
in $\cellcomplex(\varphi)$ is the edge path, $\tilde v_+$,
spanned by the following $n_+-n+1$ vertices:
\[
\{(3m+1,n)\}\cup\{(3m+2,k) \mid n+1\le k\le n_+\}.
\]
Similarly, the inverse image of $v_-$
in $\cellcomplex(\varphi)$ is the edge path, $\tilde v_-$,
spanned by the following $n-n_-+1$ vertices:
\[
\{(3m+1,n_-)\}\cup\{(3m+2,k)\mid n_-+1\le k\le n\}.
\]
Furthermore, the vertices $v_+$ and $v_-$
are contained in the layer $\LL_n$ of $\Delta(\varphi)$
in this order.
\item
$\alpha$ has $n-n_-+1$ vertices $\{P_{2m,k}\mid n_-\le k\le n\}$,
and it is the homeomorphic image of the edge path, $\tilde\alpha$,
in $\cellcomplex(\varphi)$ with vertex set:
\[
\{(3m+1,k)\mid n_-\le k\le n\}.
\]
\item
$\beta$ has $n_+-n+1$ vertices $\{P_{2m+1,k} \mid n\le k\le n_+\}$,
and it is the homeomorphic image of the edge path, $\tilde\beta$, in
$\cellcomplex(\varphi)$ with vertex set:
\begin{gather*}
\{(3m+2,n)\}\cup\{(3m+3,k)\mid n+1\le k\le n_+-1\}\\
\cup \{(3m+2,n_+)\}.
\end{gather*}
\end{enumerate}

{\normalfont(2)} Suppose $f_n=L$.
Then the open $2$-cell $c_{m,n}$ of $CW(\varphi)$
has the color gray and its boundary is the union of
an edge path, $\alpha$, of $\partial_{2m}$ and
an edge path, $\beta$, of $\partial_{2m-1}$
which satisfy the following conditions.
\begin{enumerate}[\normalfont (i)]
\item
$\alpha\cap\beta=\partial\alpha=\partial\beta$
and it consists of the following two vertices.
\begin{align*}
& v_+:=P_{2m,n}=P_{2m-1,n_+},\\
& v_-:=P_{2m,n_-}=P_{2m-1,n_-}.
\end{align*}
Moreover the inverse image of $v_+$
in $\cellcomplex(\varphi)$ is the edge path, $\tilde v_+$,
spanned by the following $n_+-n+1$ vertices:
\[
\{(3m+1,n)\}\cup\{(3m,k)\mid n+1\le k\le n_+\}.
\]
Similarly, the inverse image of $v_-$
in $\cellcomplex(\varphi)$ is the edge path, $\tilde v_-$,
spanned by the following $n-n_-+1$ vertices:
\[
\{(3m+1,n_-)\}\cup\{(3m,k)\mid n_-+1\le k\le n\}.
\]
Furthermore, the vertices $v_-$ and $v_+$
are contained in the layer $\LL_n$ of $\Delta(\varphi)$
in this order.
\item
$\alpha$ has $n-n_-+1$ vertices $\{P_{2m,k}\mid n_-\le k\le n\}$,
and it is the homeomorphic image of the edge path, $\tilde\alpha$,
in $\cellcomplex(\varphi)$ with vertex set:
\[
\{(3m+1,k)\mid n_-\le k\le n\}.
\]
\item
$\beta$ has $n_+-n+1$ vertices $\{P_{2m-1,k} \mid n\le k\le n_+\}$,
and it is the homeomorphic image of the edge path, $\tilde\beta$, in
$\cellcomplex(\varphi)$ with vertex set:
\[
\{(3m,n)\}\cup\{(3m-1,k)\mid n+1\le k\le n_+-1\} \cup \{(3m,n_+)\}.
\]
\end{enumerate}
\end{lemma}

\begin{proof}
By the definition of $CW^*(\varphi)$ and by Remark
\ref{remark.CW-vertices}, the boundary  of $c_{m,n}$ is the union of
an edge path $\alpha$ in $\partial_{2m}$ and $\beta$ in
$\partial_{2m\pm 1}$ such that
$\alpha\cap\beta=\partial\alpha=\partial\beta$,   the vertex set
of $\alpha$ is $\{P_{2m,k}\mid n_-\le k\le n\}$, and   the
vertex set of $\beta$ is $\{P_{2m\pm 1,k} \mid n\le k\le n_+\}$.
Here the signs $\pm$ stand for $+$ or $-$ according as $f_n=R$ or
$L$. By Lemma \ref{lemma.P(m,n)vsQ(m,n)2}, we see that the condition
(i) is satisfied, except the last statement that $v_+$ and $v_-$ are
contained in the layer $\LL_n$ of $\Delta(\varphi)$ in this order.
The last statement in (i), for the case $f_n=R$, follows from the
fact that $v_+$ and $v_-$, respectively, are the images of the
vertices $(3m+1,n)$ and $(3m+2,n)$ of $\cellcomplex(\varphi)$ and
the fact that the layer $\LL_n$ is the image of the horizontal line
of height $n$ in $\cellcomplex(\varphi)$. The statement for the case
$f_n=L$ is proved similarly. The remaining conditions (ii) and (iii)
follow from Lemma \ref{lemma.P(m,n)vsQ(m,n)}(1).
\end{proof}

In the above lemma,
the union of the edge path $\tilde\alpha$, $\tilde\beta$,
$\tilde v_+$ and $\tilde v_-$
forms a simple closed edge path in $\cellcomplex(\varphi)$.
Let $\tilde c_{m,n}$
be the $2$-cell bounded by it
(see Figure \ref{fig.lifted-c(m,n)}).
Consider first the case where
the $2$-cell $c_{m,n}$ is colored white (i.e.,
$f_n=R$).
Then the horizontal edges of $\cellcomplex(\varphi)$
contained in $\tilde c_{m,n}$
are as follows:
\begin{align*}
& \langle  (3m+1,k),(3m+2,k)\rangle \quad  (n_-+1\le k\le n),\\
& \langle  (3m+2,k),(3m+3,k)\rangle \quad  (n\le k\le n_+-1).
\end{align*}
By using Lemma \ref{lemma.2-cell-boundary},
we see that the images of these edges in $c(m,n)$
constitute the following three families of arcs,
which have mutually disjoint interiors:
\begin{itemize}
\item
an arc joining $v_-$ and $v_+$,
\item
arcs joining $v_-$ with the vertices contained in the interior of $\alpha$,
\item
arcs joining $v_+$ with the vertices contained in the interior of $\beta$.
\end{itemize}
The same result holds
when the color of $c_{m,n}$ is white.
This observation leads us to the following recipe
for recovering $\Delta(\varphi)$
from the colored $\langle D\rangle$-CW-complex $CW(\varphi)$.

\textit{
Let $c$ be a $2$-cell of $CW(\varphi)$.
Then it contains exactly two vertices,
$v_+$ and $v_-$, which share the same color with $c$.
We assume $v_+$ lies \lq above' $v_-$.
These two vertices divide $\partial c$
into two arcs, $\alpha$ and $\beta$.
We assume $\alpha$ lies on the left or right hand side of $c$
according as the color of $c$ is white or gray.
In the $2$-cell $c$, we draw the following families of arcs
with mutually disjoint interiors
$($see Figure $\ref{fig.lifted-c(m,n)})$:
\begin{itemize}
\item
an arc joining $v_-$ and $v_+$,
\item
arcs joining $v_-$ with
the vertices contained in the interior of $\alpha$,
\item
arcs joining $v_+$ with the vertices contained in the interior of $\beta$,
\item
arcs joining consecutive vertices of $\alpha$ not already
joined by the foregoing,
\item arcs joining consecutive vertices of $\beta$ not already
joined by the foregoing.
\end{itemize}
Next, shrink each bigon
into
a straight
edge until there are no bigons.}

\textit{
The number of bigons equals the number of vertices and is $2$ more than the
number of triangles $($see Figure $\ref{fig.lifted-c(m,n)})$.
The  arc joining  $v_-$ and $v_+$ is called
the {\normalfont central edge} associated with the $2$-cell $c$.}

\textit{
Perform the above operation at every $2$-cell of $CW(\varphi)$.
Then it follows from the preceding observation that
the resulting $CW$-decomposi\-tion of
$\CC$ is combinatorially isomorphic to $\Delta(\varphi)$.
}

\begin{figure}[t!]
\begin{center}
\setlength{\unitlength}{1truecm}
\begin{picture}(12,9)
\put(2,1){\epsfig{file=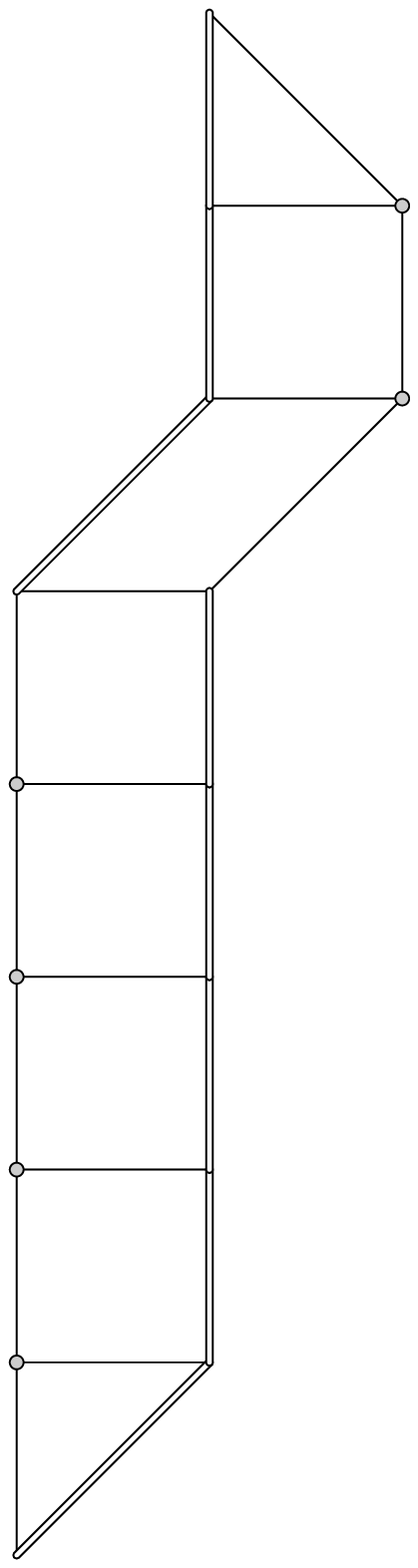, width =2.1truecm}}
\put(3.7,5){\vector(1,0){1.7}}
\put(6,3){\epsfig{file=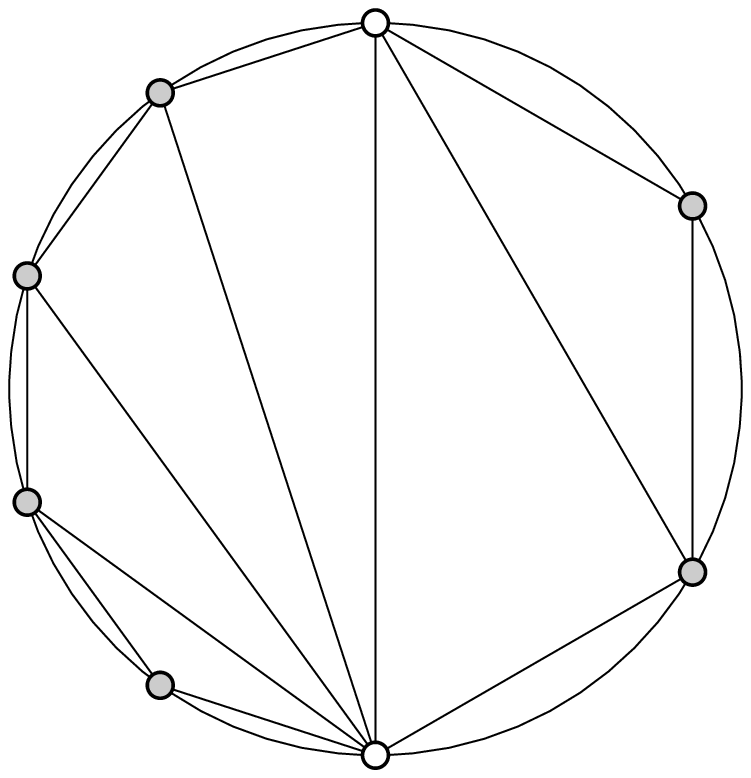, width =5truecm}}
\put(2.7,7.4){\makebox(0,0)[c]{$\tilde v_+$}}
\put(4.3,7.4){\makebox(0,0)[c]{$\tilde \beta$}}
\put(8.5,8.3){\makebox(0,0)[c]{$ v_+$}}
\put(3.4,4.1){\makebox(0,0)[c]{$\tilde v_-$}}
\put(1.7,4.1){\makebox(0,0)[c]{$\tilde \alpha$}}
\put(5.8,5.6){\makebox(0,0)[c]{$\alpha$}}
\put(11.1,5.6){\makebox(0,0)[c]{$\beta$}}
\put(8.6,2.9){\makebox(0,0)[c]{$ v_-$}}
\put(2.8,1){\makebox(0,0)[c]{$\tilde c_{m,n}$}}
\put(8.4,2.3){\makebox(0,0)[c]{$c_{m,n}$}}
\put(2,9){\makebox(0,0)[c]{$\scriptstyle 3m+1$}}
\put(3,9){\makebox(0,0)[c]{$\scriptstyle 3m+2$}}
\put(4.1,9){\makebox(0,0)[c]{$\scriptstyle 3m+3$}}
\put(0.2,1.2){\makebox(0,0)[l]{$\scriptstyle n_-$}}
\put(0.2,2.1){\makebox(0,0)[l]{$\scriptstyle n_-+1$}}
\put(0.2,3.1){\makebox(0,0)[l]{$\scriptstyle n_-+2$}}
\put(0.4,3.6){\makebox(0,0)[c]{$\scriptstyle \vdots$}}
\put(0.2,4){\makebox(0,0)[l]{$\scriptstyle n-2$}}
\put(0.2,5){\makebox(0,0)[l]{$\scriptstyle n-1$}}
\put(0.2,6){\makebox(0,0)[l]{$\scriptstyle n$}}
\put(0.2,7){\makebox(0,0)[l]{$\scriptstyle n+1$}}
\put(0.4,7.5){\makebox(0,0)[c]{$\scriptstyle \vdots$}}
\put(0.2,7.9){\makebox(0,0)[l]{$\scriptstyle n_+-1$}}
\put(0.2,8.7){\makebox(0,0)[l]{$\scriptstyle n_+$}}
\end{picture}
\caption{The $2$-cell $c_{m,n}$ in $CW(\varphi)$ and
the chosen region
$\tilde c_{m,n}$ in $\cellcomplex(\varphi)$}\label{fig.lifted-c(m,n)}
\end{center}
\end{figure}

In order to reconstruct the layered structure $\LL_n$ of $\Delta(\varphi)$,
we need the following observation.
Let $v$ be a vertex of $\Delta(\varphi)$, and let
$$\LL_p, \LL_{p+1}, \cdots, \LL_q$$
 be the layers of $\Delta(\varphi)$
containing $v$.
For each $n$ ($p\le n\le q$),
let $e_n^{(\ell)}$, resp.  $e_n^{(r)}$,  be the edge  of $\LL_n$
which is incident on $v$ and
lies on the left-hand, resp. right-hand, side of $v$.
Then the edges incident on $v$ are located around $v$
in the following clockwise cyclic order:
\begin{equation}
\label{edge-cyclic-order}
e_p^{(\ell)}, e_{p+1}^{(\ell)}, \cdots, e_q^{(\ell)},
e_q^{(r)}, e_{q-1}^{(r)}, \cdots, e_p^{(r)}.
\end{equation}
We call
$e_p^{(\ell)}$, $e_p^{(r)}$,
$e_q^{(\ell)}$ and $e_q^{(r)}$, respectively,
the \textit{lower-left edge}, the \textit{lower-right edge},
the \textit{upper-left edge}, and the \textit{upper-right edge} of $v$.
If we could identify one of the above four \lq characteristic' edges of $v$,
then we
would
know the \lq local structure' of the layered structure
around the vertex
$v$, because
the information on the characteristic edge together with
the cyclic order (\ref{edge-cyclic-order})
tells us
how the edges incident on $v$
are paired by the layered structure.
In fact, if we could identify, say the lower-left edge $e_p^{(\ell)}$, of
$v$, then
the cyclic order (\ref{edge-cyclic-order})
extends
to an order, by regarding $e_p^{(\ell)}$ as the first edge,
such that the $i$-th edge and $(2d+1-i)$-th edge belong to the same layer for
each $i$ ($1\le i\le d$), where
$d=q-p+1$ (and hence $2d$ is the valence of $v$).
The following lemma enables us to identify the
four \lq characteristic' edges of
each of the vertices of $CW(\varphi)$
(see Figure \ref{fig.local-layered-structure}).

\begin{figure}[t!]
\begin{center}
\setlength{\unitlength}{1truecm}
\begin{picture}(6,8.9)
\put(0,0){\epsfig{file=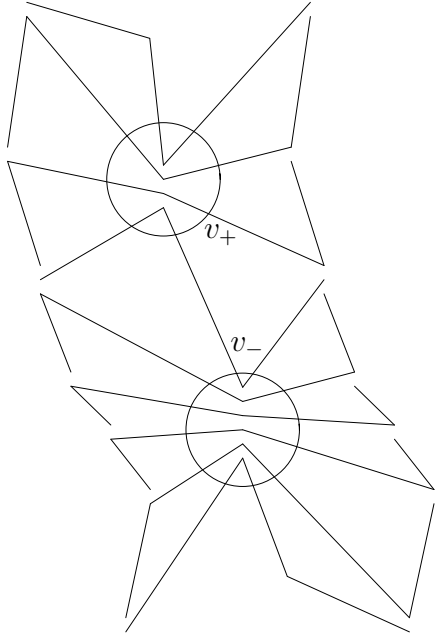, width =6truecm}}
\put(3,5.5){\makebox(0,0)[c]{$v_+$}}
\put(3.35,4){\makebox(0,0)[c]{$v_-$}}
\end{picture}
\caption{The local layered structure around the vertices $v_+$ and $v_-$}\label{fig.local-layered-structure}
\end{center}
\end{figure}

\begin{lemma}
\label{lemma.layer}
Under the setting of Lemma \ref{lemma.2-cell-boundary},
the following holds.
If $c_{m,n}$ is a white $2$-cell,
then the central edge
$\langle v_+,v_-\rangle$
is the upper-left edge of $v_-$ and
is the lower-right edge of $v_+$.
If $c_{m,n}$ is a gray $2$-cell,
then the central edge $\langle v_-,v_+\rangle$
is the upper-right edge of
$v_-$
and
is the lower-left edge of
$v_+$.
\end{lemma}

\begin{proof}
We prove the lemma when $c_{m,n}$ is a white $2$-cell.
The other case is proved similarly.
So, assume $c_{m,n}$ is a white $2$-cell.
Then $v_+$ and $v_-$, respectively,
are the images of the vertices
$(3m+1,n)$ and $(3m+2,n)$ of
$\cellcomplex(\varphi)$,
and the layer $\LL_n$ of $\Delta(\varphi)$ contains the edge
$\langle v_+,v_-\rangle$
by Lemma \ref{lemma.2-cell-boundary}.
Thus the image, $w_+$, of $(3m,n)$ is the predecessor of $v_+$ in $\LL_n$,
and the image, $w_-$, of $(3m+3,n)$ is the successor of $v_-$ in $\LL_n$.

We first
show that $\langle v_+,v_-\rangle$
is the lower-right edge of $v_+$.
Observe that there is no slanted edge of $\cellcomplex(\varphi)$
which has endpoint $(3m+1,n)$ and which is contained in $\RR\times [n-1,n]$
(cf. Definition \ref{definition.parents-complex}
and Figure \ref{fig.proof-layer-lemma}).
Consider a small half circle $\tilde c$ with center $(3m+1,n)$
contained in $\RR\times [n-1,n]$
with endpoints  $(3m+1\pm\epsilon, n)$ for some
small positive real number $\epsilon$.
Then the interior of $c$ is contained in an open $2$-cell of
$\cellcomplex(\varphi)$ which has
a horizontal edge of height $n-1$
and has the following horizontal edges:
\[
\langle (3m,n), (3m+1,n)\rangle,
\langle (3m+1,n), (3m+2,n)\rangle.
\]
Thus $\tilde c$ projects to an arc $c$ around the vertex $v_+$
such that (i) $c$ is contained in a triangle of $\Delta(\varphi)$,
(ii) $\langle w_+, v_+\rangle$ and $\langle v_+, v_-\rangle$
are edges of the triangle and each of them contains an
endpoint of $c$, and that
(iii) the remaining edge of the triangle
belongs to the layer $\LL_{n-1}$.
Then it follows that
$\langle w_+, v_+, v_-\rangle$ is the triangle of
$\Delta(\varphi)$ containing $c$
and that the edge $\langle w_+, v_-\rangle$ belongs to
the layer
$\LL_{n-1}$. Hence $\langle v_+,v_-\rangle$
is the lower-right edge of $v_+$.

Next, we show that $\langle v_+,v_-\rangle$ is the upper-left edge of $v_-$.
To this end, note that the assumption $f_n=R$ implies that
$\cellcomplex(\varphi)$ has the
following two slanted edges
(see Figure \ref{fig.proof-layer-lemma}):
\[
\langle (3m+1,n), (3m+2,n+1)\rangle,
\langle (3m+3,n), (3m+3,n+1)\rangle.
\]
Observe that these two slanted edges and the
three horizontal edges
\begin{align*}
&\langle (3m+1,n), (3m+2,n)\rangle,
\langle (3m+2,n), (3m+3,n)\rangle,\\
&\langle (3m+2,n+1), (3m+3,n+1)\rangle
\end{align*}
bound a $2$-cell of $\cellcomplex(\varphi)$.
This implies that
$\langle v_+, v_-, w_-\rangle$
is a $2$-simplex of $\Delta(\varphi)$
and that $\langle v_+, w_-\rangle$ is an edge
of the layer $\LL_{n+1}$.
Hence $\langle v_+,v_-\rangle$ is the upper-left edge of $v_-$.
\end{proof}

\begin{figure}[t!]
\begin{center}
\setlength{\unitlength}{1truecm}
\begin{picture}(11.5,2.7)
\put(0.15,0.2){\epsfig{file=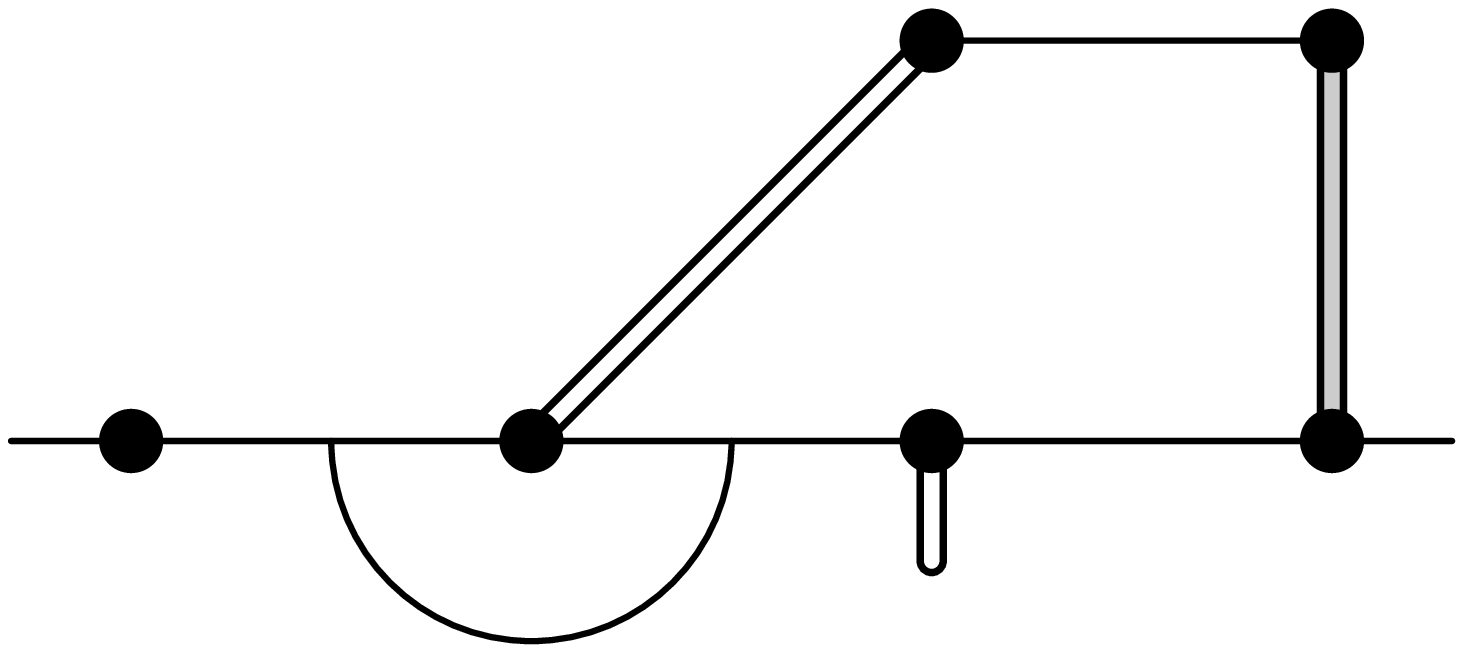, width =5truecm}}
\put(6,0.2){\epsfig{file=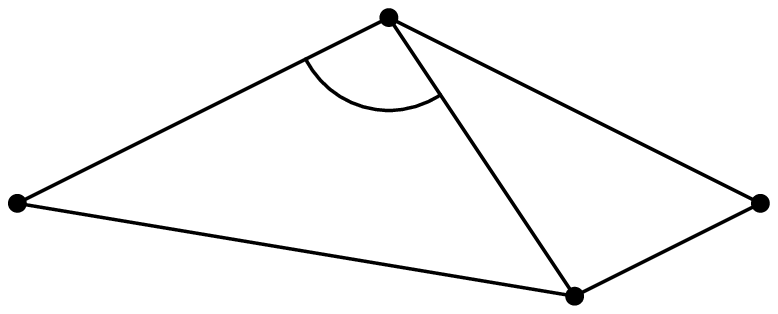, width =5truecm}}
\put(.7,2.6){\makebox(0,0)[c]{$\scriptstyle 3m$}}
\put(2.05,2.6){\makebox(0,0)[c]{$\scriptstyle 3m+1$}}
\put(3.25,2.6){\makebox(0,0)[c]{$\scriptstyle 3m+2$}}
\put(4.65,2.6){\makebox(0,0)[c]{$\scriptstyle 3m+3$}}
\put(0,.9){\makebox(0,0)[r]{$\scriptstyle n$}}
\put(0,2.3){\makebox(0,0)[r]{$\scriptstyle n+1$}}
\put(2.1,0){\makebox(0,0)[r]{$\tilde c$}}
\put(10.5,1.7){\makebox(0,0)[r]{$\LL_{n+1}$}}
\put(8.1,0.4){\makebox(0,0)[r]{$\LL_{n-1}$}}
\put(8.8,2.28){\makebox(0,0)[r]{$v_+$}}
\put(8.4,1.3){\makebox(0,0)[r]{$c$}}
\put(11.5,.9){\makebox(0,0)[r]{$w_-$}}
\put(6.05,.9){\makebox(0,0)[r]{$w_+$}}
\put(10,0){\makebox(0,0)[r]{$v_-$}}
\end{picture}
\caption{The layer $\LL_n$ passing through $v_+$ and $v_-$}
\label{fig.proof-layer-lemma}
\end{center}
\end{figure}

The above lemma leads us to the following recipe
for recovering the layered structure $\{\LL_n\}_{n\in\ZZ}$ of
$\Delta(\varphi)$ from the
colored $\langle D\rangle$-CW-complex $CW(\varphi)$.

\textit{
Let $v$ be a vertex of $\Delta(\varphi)$.
Since it is also a vertex of $CW(\varphi)$, it has a color white or gray.
Let $e_1,e_2,\cdots, e_{2d}$ be the edges of the $\Delta(\varphi)$
incident on $v$ in counter-clockwise or clockwise order
according as $v$ is white or gray.
By the construction of $CW(\varphi)$,
there are precisely two $2$-cells
of $CW(\varphi)$ which have $v$ as a vertex
and share the same color with $v$.
Let $c$ be such
a
$2$-cell which lies below $v$, and let
$\langle v_+, v_-\rangle$ be the central edge associated with $c$,
where $v=v_+$.
We may assume, after cyclic permutation,
that $e_1=\langle v_+, v_-\rangle$.
Then the couplings
$\{(e_i,e_{2d-i+1})\}_{1\le i\le d}$
gives the desired \lq\lq local layered structure around $v$''.
$($In fact, $e_1$ is the lower-right or upper-right vertex of $v$
according as $v$ is white or gray by
Lemma $\ref{lemma.layer}$.
So the above local layered structure at each vertex $v$
is consistent with the layered structure $\{\LL_n\}$ of $\Delta(\varphi).)$
By combining this local information
at the vertices of
$\Delta(\varphi)$, we obtain the layered structure.
}
Thus we have shown that
the combinatorial structure of
the layered $\langle D \rangle$-simplicial complex
$(\Delta(\varphi), \{\LL_n\})$
can be recovered from that of
the colored
$\langle D \rangle$-CW-complex $CW(\varphi)$.
This completes the proof of Theorem \ref{maintheorem}.

\begin{remark}
\rm
The main theorem, Theorem \ref{maintheorem},
is actually valid for all
doubly degenerate punctured-torus groups
with
\lq\lq bounded geometry''.
(See the celebrated paper \cite{Minsky} by Minsky,
for the classification of punctured-torus groups.)
In fact, the description of the Cannon-Thurston maps
given by Bowditch \cite{Bowditch} is valid for every
such group and thus a fractal tessellation of the complex plane
is naturally associated with the group, for which
an analogue of Theorem \ref{Th.CDTessellation} holds.
On the other hand, the canonical decompositions of
the quotient hyperbolic manifolds
associated with the punctured-torus groups
are determined by Akiyoshi \cite{Akiyoshi} and
Gueritaud \cite{Gueritaud2}.
In particular, an analogue of Theorem \ref{thm.canonical-decomposition}
holds for all punctured-torus groups.
These two results guarantee that
the proof of Theorem \ref{maintheorem} works
for all doubly degenerate punctured-torus groups
with bounded geometry,
and hence we have
an analogue of Theorem \ref{maintheorem} for such groups.
\end{remark}

\end{document}